\documentclass[11pt, twoside]{article}
\usepackage{amsmath,amsthm,amssymb}
\usepackage{times}
\usepackage{enumerate}
\usepackage{amssymb}
\usepackage{graphicx}
\usepackage[noadjust]{cite}
\usepackage{color}

\def\titlerunning#1{\gdef\titrun{#1}}
\makeatletter
\def\author#1{\gdef\autrun{\def\and{\unskip, }#1}\gdef\@author{#1}}

\makeatother

\def\keywords#1{\par\medskip
\noindent\textbf{Keywords.} #1}
\def\subjclass#1{\par\smallskip
\noindent\textbf{MSC (2010):} #1}

\newtheorem{thm}{Theorem}[section]

\newtheorem{lem}[thm]{Lemma}

\newtheorem{prop}[thm]{Proposition}

\theoremstyle{definition}

\newtheorem{rem}[thm]{Remark}

\numberwithin{equation}{section}

\newtheorem*{notations}{Notations}

\DeclareMathOperator*{\esssup}{ess\,sup}
\DeclareMathOperator*{\essinf}{ess\,inf}

\makeatletter
\let\@fnsymbol\@alph
\makeatother

\begin{document}

\baselineskip=17pt

\titlerunning{Compactness for $p$-Laplacian}

\title{Compactness results for the $p$-Laplace equation}

\author{Marino Badiale\thanks{Dipartimento di Matematica ``Giuseppe Peano'', Universit\`{a} degli Studi di
Torino, Via Carlo Alberto 10, 10123 Torino, Italy. 
e-mails: \texttt{marino.badiale@unito.it}, \texttt{michela.guida@unito.it}}
\textsuperscript{,}\thanks{Partially supported by the PRIN2012 grant ``Aspetti variazionali e
perturbativi nei problemi di.renziali nonlineari''.}
\ -\ Michela Guida\textsuperscript{a,}\thanks{Member of the Gruppo Nazionale di Alta Matematica (INdAM).}
\ -\ Sergio Rolando\thanks{Dipartimento di Matematica e Applicazioni, Universit\`{a} di Milano-Bicocca,
Via Roberto Cozzi 53, 20125 Milano, Italy. e-mail: \texttt{sergio.rolando@unito.it}}%
\ \textsuperscript{,\,c}
}
%Marino Badiale\thanks{Partially supported by the PRIN2009 grant ``Critical Point Theory and
%Perturbative Methods for Nonlinear Differential Equations''},
%\ \ Michela Guida,
%\ \ Sergio Rolando\footnotemark[1]}

\date{
%\begin{footnotesize}
%\emph{
%Dipartimento di Matematica ``Giuseppe Peano'' \smallskip\\
%Universit\`{a} degli Studi di Torino, Via Carlo Alberto 10, 10123 Torino, Italy \\
%e-mail:} marino.badiale@unito.it, michela.guida@unito.it, sergio.rolando@unito.it
%\end{footnotesize}
}
\maketitle

\begin{abstract}
Given $1<p<N$ and two measurable functions $V\left(r \right)\geq 0$ and $K\left(r\right)> 0$, $r>0$,
we define the weighted spaces 
\[
W=\left\{ u\in D^{1,p}(\mathbb{R}^{N}):\int_{\mathbb{R}^{N}}V\left( \left|
x\right| \right) \left| u\right|^{p}dx<\infty \right\} ,\quad L_{K}^{q}=L^{q}(\mathbb{R}%
^{N},K\left( \left| x\right| \right) dx)
\]
and study the compact embeddings of the radial subspace of $W$ into $%
L_{K}^{q_{1}}+L_{K}^{q_{2}}$, and thus into $L_{K}^{q}$ ($%
=L_{K}^{q}+L_{K}^{q}$) as a particular case. 
%Both super- and sub-quadratic exponents $q_{1}$, $q_{2}$ and $q$ are considered.
Both exponents $q_{1},q_{2},q$ greater and lower than $p$ are considered.
Our results do not require any compatibility between how the potentials $V$ and $K$ behave at the origin and at infinity, and 
essentially rely on power type estimates of their relative growth, not of the potentials separately.
%Our results concern both super- and sub-quadratic exponents $q_{1}$, $q_{2}$ and $q$, and essentially rely on power type estimates of the relative growth of the potentials $V$ and $K$ at the origin and at infinity.

\end{abstract}

\section{Introduction}

In this paper we pursue the work we made in papers \cite{BGR_I,BGR_II,GR-nls}, 
where we studied embedding and compactness results for weighted Sobolev spaces. These results then made possible to get existence and multiplicity results, by variational methods, for semilinear elliptic equations in $\mathbb{R}^N$. 
%\par \noindent 

In the present paper we face nonlinear elliptic $p$-Laplace equations, that is,
\begin{equation}
-\triangle_p u+V\left( \left| x\right| \right) |u|^{p-1}u=K\left( \left| x\right|
\right) f\left( u\right) \quad \text{in }\mathbb{R}^{N}.  \label{EQ}
\end{equation}
Here $1<p<N$, $f:\mathbb{R}\rightarrow \mathbb{R}$ is a continuous nonlinearity
satisfying $f\left( 0\right) =0$ and $V\geq 0,K>0$ are given potentials.

%\noindent 
To study this problem we introduce the space
\[
W:=\left\{ u\in D^{1,p}(\mathbb{R}^{N}):\int_{\mathbb{R}^{N}}V\left( \left|
x\right| \right) |u|^{p}dx<\infty \right\}
% ,\quad L_{K}^{q}=L^{q}(\mathbb{R}^{N},K\left( \left| x\right| \right) dx)
\]
%\noindent 
equipped with the standard norm
$$\left\| u\right\|^p:=\int_{\mathbb{R}^{N}} \left(\left| \nabla u\right| ^{p}
+V\left( \left| x\right| \right) |u|^{p}\right)dx,$$
%
%\noindent in $W$ and
%
%$$|u|_K^q = \int_{\mathbb{R}^{N}} K\left( \left| x\right| \right) u^{q}dx$$
%
%\noindent in $ L_{K}^{q}$.
%\par \noindent We 
and say that $u\in W$ is a \textit{weak solution}\emph{\ }to (\ref{EQ}) if 
\begin{equation}
\int_{\mathbb{R}^{N}}|\nabla u|^{p-2}\nabla u\cdot \nabla h\,dx+\int_{\mathbb{R}^{N}}V\left(
\left| x\right| \right) |u|^{p-2}uh\,dx=\int_{\mathbb{R}^{N}}K\left( \left| x\right|
\right) f\left( u\right) h\,dx\quad \textrm{for all }h\in W.
\label{weak solution}
\end{equation}
%\noindent 
The natural approach in studying weak solutions to equation (\ref{EQ}) is variational, since
these solutions are (at least formally) critical points of the Euler
functional 
\begin{equation}
I\left( u\right) :=\frac{1}{p}\left\| u\right\| ^{p}-\int_{\mathbb{R}%
^{N}}K\left( \left| x\right| \right) F\left( u\right) dx,  \label{I:=}
\end{equation}
where $F\left( t\right) :=\int_{0}^{t}f\left( s\right) ds$. Then the problem
of existence is easily solved if $V$ does not vanish at infinity and $K$ is
bounded, because standard embeddings theorems of $W$ and its radial subspace into the weighted Lebesgue space
\[
L_{K}^{q}:=L_{K}^{q}(\mathbb{R}^{N}):=L^{q}(\mathbb{R}^{N},K\left( \left| x\right| \right) dx)
\]
are available (for suitable $q$'s). As we let $V$ and $K$ to vanish, 
or to go to infinity, as $|x| \rightarrow 0$ or $|x|\rightarrow +\infty$, the usual embeddings theorems for Sobolev spaces are not 
available anymore, and new embedding theorems need to be proved. This has been done in several papers: see e.g. the references in 
\cite{BGR_I,BGR_II,GR-nls} for a bibliography concerning the usual Laplace equation, 
and \cite{Anoop,Su12,Cai-Su-Sun,SuTian12,Su-Wang-Will-p,Yang-Zhang,Zhang13,BPR}
for equations involving the $p$-laplacian.

%\par \noindent 
The main novelty of our approach (in \cite{BGR_I,BGR_II} and in the present paper) is two-fold. First, we look for embeddings of $W_r$ (the radial subspace of $W$) not into a single Lebesgue space $L_{K}^{q}$ %:=L_{K}^{q}\left( \mathbb{R}^{N}\right)$ 
but into a sum of Lebesgue spaces $L_{K}^{q_1}+L_{K}^{q_2}$. This allows to study separately the behaviour of the potentials $V,K$ at $0$ and $\infty$, and to assume different set of hypotheses about these behaviours. Second, we assume hypotheses not on $V$ and $K$ separately but on their ratio, so allowing asymptotic behaviours of general kind for the two potentials.

%\par \noindent 
Thanks to these novelties, our embedding results yield existence of solutions for (\ref{EQ}) in cases 
which are not covered by the previous literature. 
Moreover, one can check that our embeddings are also new in some of the cases already treated in previous papers, thus giving existence results which improve some well-known theorems in the literature.
%
%for (\ref{EQ}) which cover cases where the results in the previous literature didn't say anything. %said nothing. 
%Moreover, but our theorems give some new results also in some of the cases already treated in the previous literature.
%
%we get embedding results of $W_r$ into  $L_{K}^{q_1}+L_{K}^{q_2}$ 
%(and thus into $L_{K}^{q}=L_{K}^{q}+L_{K}^{q}$ as a particular case), 
%and hence existence of solutions for (\ref{EQ}), also covering cases where the results in the previous literature 
%did not say anything.
%said nothing. but our theorems give some new results also in some of the cases already treated in the previous literature.

%\par \noindent 
In the present paper we limit ourselves to the proof of the compact embeddings, which is the hardest part of the arguments. In a forthcoming paper \cite{BGR_p} we will apply these results to get existence and multiplicity results for $p$-laplacian equations like   (\ref{EQ}).

%\par \noindent 
This paper is organized as follows. In Section \ref{SEC:MAIN} we state our
main results: a general result concerning the embedding properties of $W_r$  into $L_{K}^{q_{1}}+L_{K}^{q_{2}}$ (Theorem \ref{THM(cpt)})
and some explicit conditions ensuring that the embedding is compact
(Theorems \ref{THM0}, \ref{THM1}, \ref{THM2} and \ref{THM3}). The general
result is proved in Section \ref{SEC:1}, the explicit conditions in Section 
\ref{SEC:2}. 
The Appendix is devoted to some detailed computations, displaced from Section \ref{SEC:2} for sake of
clarity.%\smallskip

\begin{notations}
%\noindent \textbf{Notations. }
We end this introductory section by collecting
some notations used in the paper.%\smallskip

\noindent $\bullet $ For every $R>0$, we set $B_{R}:=\left\{ x\in \mathbb{R}%
^{N}:\left| x\right| <r\right\} $.

\noindent $\bullet $ For any subset $A\subseteq \mathbb{R}^{N}$, we denote $%
A^{c}:=\mathbb{R}^{N}\setminus A$. If $A$ is Lebesgue measurable, $\left|
A\right| $ stands for its measure.

%\noindent $\bullet $ $O\left( N\right) $ is the orthogonal group of $\mathbb{R}^{N}$.

\noindent $\bullet $ By $\rightarrow $ and $\rightharpoonup $ we
respectively mean \emph{strong} and \emph{weak }convergence.

\noindent $\bullet $ $\hookrightarrow $ denotes \emph{continuous} embeddings.

\noindent $\bullet $ $C_{\mathrm{c}}^{\infty }(\Omega )$ is the space of the
infinitely differentiable real functions with compact support in the open
set $\Omega \subseteq \mathbb{R}^{d}$.
%; $C_{\mathrm{c,rad}}^{\infty }(\mathbb{R}^{N})$ is the radial subspace of $C_{\mathrm{c}}^{\infty }(\mathbb{R}^{N})$.

\noindent $\bullet $ If $1\leq p\leq \infty $ then $L^{p}(A)$ and $L_{%
\mathrm{loc}}^{p}(A)$ are the usual real Lebesgue spaces (for any measurable
set $A\subseteq \mathbb{R}^{d}$). If $\rho :A\rightarrow \left( 0,+\infty
\right) $ is a measurable function, then $L^{p}(A,\rho \left( z\right) dz)$
is the real Lebesgue space with respect to the measure $\rho \left( z\right)
dz$ ($dz$ stands for the Lebesgue measure on $\mathbb{R}^{d}$).

\noindent $\bullet $ $p^{\prime }:=p/(p-1)$ is the H\"{o}lder-conjugate
exponent of $p.$

\noindent$\bullet $  For $1<p<N$, $D^{1,p}(\mathbb{R}^{N})=\{u\in L^{p^{*}}(\mathbb{R}%
^{N}):\nabla u\in L^{2}(\mathbb{R}^{N})\}$ %, $N\geq 3$, 
is the usual Sobolev space, which identifies with the completion of $C_{\mathrm{c}}^{\infty }(%
\mathbb{R}^{N})$ with respect to the norm of the gradient; $D_{\mathrm{rad}%
}^{1,p}(\mathbb{R}^{N})$ is the radial subspace of $D^{1,p}(\mathbb{R}^{N})$; $%
D_{0}^{1,p}\left( B_{R}\right) $ is closure of $C_{\mathrm{c}}^{\infty
}\left( B_{R}\right) $ in $D^{1,p}(\mathbb{R}^{N})$.

\noindent  $\bullet $  For $1<p<N$, $p^{*}:=pN/\left( N-p\right) $ is the critical exponent
for the Sobolev embedding in dimension $N$. %\geq 3$.

\end{notations}

\section{Main results \label{SEC:MAIN}}

\par \noindent We consider %$N\geq 3$, 
$1<p<N$ and we assume the following hypotheses on $V,K$:

\begin{itemize}
\item[$\left( \mathbf{V}\right) $]  $V:\left( 0,+\infty \right) \rightarrow
\left[ 0,+\infty \right] $ is a measurable function such that $V\in
L^{1}\left( \left( r_{1},r_{2}\right) \right) $ for some $r_{2}>r_{1}>0;$

\item[$\left( \mathbf{K}\right) $]  $K:\left( 0,+\infty \right) \rightarrow
\left( 0,+\infty \right) $ is a measurable function such that $K\in L_{%
\mathrm{loc}}^{s}\left( \left( 0,+\infty \right) \right) $ for some 
$s>1$.
%$$s> \frac{p^*}{p^* -1} =  \frac{Np}{N(p-1)+p} \quad \left( p^* = \frac{Np}{N-p}\right).$$

\end{itemize}

\noindent Let us define the following function spaces
\begin{equation}
W:= D^{1,p}(  \mathbb{R}^N ) \cap L^p (\mathbb{R}^N,V(|x|)dx), 
\quad W_r := D^{1,p}_{\mathrm{rad}} (  \mathbb{R}^N ) \cap L^p (\mathbb{R}^N,V(|x|)dx)
\label{spaces}
\end{equation}
and let $||u||$ be the standard norm in $W$ (and $W_r$).
Assumption $\left( \mathbf{V}\right) $ implies that the spaces 
$W$ and $W_r$ are nontrivial, while hypothesis $\left( \mathbf{K}\right) $ 
ensures that $W_r$ is compactly embedded into the weighted Lebesgue space $L_{K}^{q}(B_{R}\setminus B_{r})$
for every $1<q<\infty $ and $R>r>0$ (cf. Lemma \ref{Lem(corone)} below). In
what follows, the summability assumptions in $\left( \mathbf{V}\right) $ and 
$\left( \mathbf{K}\right) $ will not play any other role than this.

Given $V$ and $K$, we define the following functions of $R>0$ and $q>1$: 
\begin{eqnarray}
\mathcal{S}_{0}\left( q,R\right)&:=&
\sup_%\Sb 
{u\in W_r,\,%  \\ %
\left\| u\right\| =1  }%\endSb 
\int_{B_{R}}K\left( \left| x\right| \right)
\left| u\right| ^{q}dx,  \label{S_o :=}
\\
\mathcal{S}_{\infty }\left( q,R\right)&:=&
\sup_%\Sb
{u\in W_r,\,%  \\ 
\left\| u\right\| =1  }%\endSb 
\int_{\mathbb{R}%
^{N}\setminus B_{R}}K\left( \left| x\right| \right) \left| u\right| ^{q}dx.
\label{S_i :=}
\end{eqnarray}
Clearly $\mathcal{S}_{0}\left( q,\cdot \right) $ is nondecreasing, $\mathcal{%
S}_{\infty }\left( q,\cdot \right) $ is nonincreasing and both of them can
be infinite at some $R$.

Our first result concerns the embedding properties of $W_r$
into $L_{K}^{q_{1}}+L_{K}^{q_{2}}$ and relies on assumptions which are quite
general, sometimes also sharp (see claim (iii)), but not so easy to check.
More handy conditions ensuring these general assumptions will be provided by
the next results. Some reference on the space $L_{K}^{q_{1}}+L_{K}^{q_{2}}$
will be given in Section \ref{SEC:1}.

\begin{thm}
\label{THM(cpt)} Let $1<p<N$, let $V$, $K$ be as in $\left( \mathbf{V}%
\right) $, $\left( \mathbf{K}\right) $ and let $q_{1},q_{2}>1$.

\begin{itemize}
\item[(i)]  If 
\begin{equation}
\mathcal{S}_{0}\left( q_{1},R_{1}\right) <\infty \quad \text{and}\quad 
\mathcal{S}_{\infty }\left( q_{2},R_{2}\right) <\infty \quad \text{for some }%
R_{1},R_{2}>0,  
\tag*{$\left( {\cal S}_{q_{1},q_{2}}^{\prime }\right) $}
\end{equation}
then $W_r$ is continuously embedded into $L_{K}^{q_{1}}(\mathbb{R}^{N})+L_{K}^{q_{2}}(\mathbb{R}^{N})$.

\item[(ii)]  If 
\begin{equation}
\lim_{R\rightarrow 0^{+}}\mathcal{S}_{0}\left( q_{1},R\right)
=\lim_{R\rightarrow +\infty }\mathcal{S}_{\infty }\left( q_{2},R\right) =0, 
\tag*{$\left({\cal S}_{q_{1},q_{2}}^{\prime \prime }\right) $}
\end{equation}
then $W_r$ is compactly embedded into $L_{K}^{q_{1}}(\mathbb{R}^{N})+L_{K}^{q_{2}}(\mathbb{R}^{N})$.

\item[(iii)]  If $K\left( \left| \cdot \right| \right) \in L^{1}(B_{1})$ and 
$q_{1}\leq q_{2}$, then conditions $\left( \mathcal{S}_{q_{1},q_{2}}^{\prime
}\right) $ and $\left( \mathcal{S}_{q_{1},q_{2}}^{\prime \prime }\right) $
are also necessary to the above embeddings.
\end{itemize}
\end{thm}

Observe that, of course, $(\mathcal{S}_{q_{1},q_{2}}^{\prime \prime })$
implies $(\mathcal{S}_{q_{1},q_{2}}^{\prime })$. Moreover, these assumptions
can hold with $q_{1}=q_{2}=q$ and therefore Theorem \ref{THM(cpt)} also
concerns the embedding properties of $W_r$ into $L_{K}^{q}$, $1<q<\infty $.
\smallskip

We now look for explicit conditions on $V$ and $K$ implying $(\mathcal{S}%
_{q_{1},q_{2}}^{\prime \prime })$ for some $q_{1}$ and $q_{2}$. More
precisely, we will ensure $(\mathcal{S}_{q_{1},q_{2}}^{\prime \prime })$
through a more stringent condition involving the following functions of $R>0$
and $q>1$: 
\begin{eqnarray}
\mathcal{R}_{0}\left( q,R\right)&:=&
\sup_%\Sb 
{
u\in H_{V,\mathrm{r}}^{1},\,h\in H_{V}^{1},\,%  \\ 
\left\| u\right\| =\left\| h\right\| =1 %\endSb
}%
\,\int_{B_{R}}K\left( \left| x\right| \right) \left| u\right| ^{q-1}\left|
h\right| dx,  \label{N_o} \\
\mathcal{R}_{\infty }\left( q,R\right)&:= &
\sup_%\Sb 
{
u\in H_{V,\mathrm{r}}^{1},\,h\in H_{V}^{1},\, \left\| u\right\| =\left\| h\right\| =1  %\endSb 
}
\,\int_{\mathbb{R}^{N}\setminus B_{R}}K\left( \left| x\right| \right) \left|
u\right| ^{q-1}\left| h\right| dx.  \label{N_i}
\end{eqnarray}
Note that $\mathcal{R}_{0}\left( q,\cdot \right) $ is nondecreasing, $%
\mathcal{R}_{\infty }\left( q,\cdot \right) $ is nonincreasing and both can
be infinite at some $R$. Moreover, for every $\left( q,R\right) $ one has $%
\mathcal{S}_{0}\left( q,R\right) \leq \mathcal{R}_{0}\left( q,R\right) $ and 
$\mathcal{S}_{\infty }\left( q,R\right) \leq \mathcal{R}_{\infty }\left(
q,R\right) $, so that $(\mathcal{S}_{q_{1},q_{2}}^{\prime \prime })$ is a
consequence of the following, stronger condition: 
\begin{equation}
\lim_{R\rightarrow 0^{+}}\mathcal{R}_{0}\left( q_{1},R\right)
=\lim_{R\rightarrow +\infty }\mathcal{R}_{\infty }\left( q_{2},R\right) =0. 
\tag*{$\left( {\cal R}_{q_{1},q_{2}}^{\prime \prime }\right) $}
\end{equation}
In Theorems \ref{THM0} and \ref{THM3} we will find ranges of exponents $%
q_{1} $ such that $\lim_{R\rightarrow 0^{+}}\mathcal{R}_{0}\left(q_{1},R\right)$ $=0$.
In Theorems \ref{THM1} and \ref{THM2} we will do the same for exponents $q_{2}$ such that
$\lim_{R\rightarrow +\infty}\mathcal{R}_{\infty }\left( q_{2},R\right) =0$.
Condition $(\mathcal{R}_{q_{1},q_{2}}^{\prime \prime })$ 
then follows by joining Theorem \ref{THM0} or \ref{THM3} with Theorem \ref{THM1} or \ref{THM2}.

%In the following results we get more explicit conditions assuring that the hypotheses of Theorem \ref{THM(cpt)} hold.
For $\alpha \in \mathbb{R}$ and $\beta \in \left[ 0,1\right] $, define two
functions $\alpha ^{*}\left( \beta \right) $ and $q^{*}\left( \alpha ,\beta
\right) $ by setting 
\[
\alpha ^{*}\left( \beta \right) :=\max \left\{ p\beta -1-\frac{p-1}{p}N ,-\left(
1-\beta \right) N\right\} =\left\{ 
\begin{array}{ll}
p\beta -1-\frac{p-1}{p}N   \quad \smallskip & \text{if }0\leq \beta \leq \frac{1}{p%
} \\ 
-\left( 1-\beta \right) N & \text{if }\frac{1}{p}\leq \beta \leq 1
\end{array}
\right. 
\]
and 
\[
q^{*}\left( \alpha ,\beta \right) :=p\frac{\alpha -p\beta +N}{N-p}. 
\]
Note that $\alpha ^{*}\left( \beta \right) \leq 0$ and $\alpha ^{*}\left(
\beta \right) =0$ if and only if $\beta =1$.

The first two Theorems \ref{THM0} and \ref{THM1} only rely on a power type
estimate of the relative growth of the potentials and do not require any
other separate assumption on $V$ and $K$ than $\left( \mathbf{V}\right) $
and $\left( \mathbf{K}\right) $, including the case $V\left( r\right) \equiv
0$ (see Remark \ref{RMK: suff12}.\ref{RMK: suff12-V^0}).

\begin{thm}
\label{THM0}
Let $1<p<N$ and let $V$, $K$ be as in $\left( \mathbf{V}\right) $, $\left( \mathbf{K}\right) $.
Assume that there exists $R_{1}>0$ such that $V\left( r\right) <+\infty $ almost everywhere in $(0,R_1)$ and
\begin{equation}
\esssup_{r\in \left( 0,R_{1}\right) }\frac{K\left( r\right) }{%
r^{\alpha _{0}}V\left( r\right) ^{\beta _{0}}}<+\infty \quad \text{for some }%
0\leq \beta _{0}\leq 1\text{~and }\alpha _{0}>\alpha ^{*}\left( \beta
_{0}\right) .  \label{esssup in 0}
\end{equation}
Then $\displaystyle \lim_{R\rightarrow 0^{+}}\mathcal{R}_{0}\left(
q_{1},R\right) =0$ for every $q_{1}\in \mathbb{R}$ such that 
\begin{equation}
\max \left\{ 1,p\beta _{0}\right\} <q_{1}<q^{*}\left( \alpha _{0},\beta
_{0}\right) .  \label{th1}
\end{equation}
\end{thm}

\begin{thm}
\label{THM1}
Let $1<p<N$ and let $V$, $K$ be as in $\left( \mathbf{V}\right) $, $\left( \mathbf{K}\right) $.
Assume that there exists $R_{2}>0$ such that $V\left( r\right) <+\infty $ for almost every $r>R_2$ and
\begin{equation}
\esssup_{r>R_{2}}\frac{K\left( r\right) }{r^{\alpha _{\infty
}}V\left( r\right) ^{\beta _{\infty }}}<+\infty \quad \text{for some }0\leq
\beta _{\infty }\leq 1\text{~and }\alpha _{\infty }\in \mathbb{R}.
\label{esssup all'inf}
\end{equation}
Then $\displaystyle \lim_{R\rightarrow +\infty }\mathcal{R}_{\infty }\left(
q_{2},R\right) =0$ for every $q_{2}\in \mathbb{R}$ such that 
\begin{equation}
q_{2}>\max \left\{ 1,p\beta _{\infty },q^{*}\left( \alpha _{\infty },\beta
_{\infty }\right) \right\} .  \label{th2}
\end{equation}
\end{thm}

We observe explicitly that for every $\left( \alpha ,\beta \right) \in \mathbb{R%
}\times \left[ 0,1\right] $ one has 
\[
\max \left\{ 1,p\beta ,q^{*}\left( \alpha ,\beta \right) \right\} =\left\{ 
\begin{array}{ll}
q^{*}\left( \alpha ,\beta \right) \quad & \text{if }\alpha \geq \alpha
^{*}\left( \beta \right) \smallskip \\ 
\max \left\{ 1,p\beta \right\} & \text{if }\alpha \leq \alpha ^{*}\left(
\beta \right)
\end{array}
\right. . 
\]

\begin{rem}
\label{RMK: suff12}\quad 

\begin{enumerate}
\item  \label{RMK: suff12-V^0}We mean $V\left( r\right) ^{0}=1$ for every $r$
(even if $V\left( r\right) =0$). In particular, if $V\left( r\right) =0$ for
almost every $r>R_{2}$, then Theorem \ref{THM1} can be applied with $\beta
_{\infty }=0$ and assumption (\ref{esssup all'inf}) means 
\[
\esssup_{r>R_{2}}\frac{K\left( r\right) }{r^{\alpha _{\infty }}}%
<+\infty \quad \text{for some }\alpha _{\infty }\in \mathbb{R}.
\]
Similarly for Theorem \ref{THM0} and assumption (\ref{esssup in 0}), if $%
V\left( r\right) =0$ for almost every $r\in \left( 0,R_{1}\right) $.

\item  \label{RMK: suff12-no hp}The inequality $\max \left\{ 1,p\beta
_{0}\right\} <q^{*}\left( \alpha _{0},\beta _{0}\right) $ is equivalent to $%
\alpha _{0}>\alpha ^{*}\left( \beta _{0}\right) $. Then, in (\ref{th1}),
such inequality is automatically true and does not ask for further
conditions on $\alpha _{0}$ and $\beta _{0}$.

\item  \label{RMK: suff12-Vbdd}The assumptions of Theorems \ref{THM0} and 
\ref{THM1} may hold for different pairs $\left( \alpha _{0},\beta_{0}\right)$,
$\left( \alpha _{\infty },\beta _{\infty }\right) $. In this
case, of course, one chooses them in order to get the ranges for $q_{1},q_{2}
$ as large as possible. For instance, if $V$ is not singular at the origin, 
i.e., $V$ is essentially bounded in a neighbourhood of 0, 
and condition (\ref{esssup in 0}) holds true for a pair $\left( \alpha _{0},\beta _{0}\right) $%
, then (\ref{esssup in 0}) also holds for all pairs $\left( \alpha
_{0}^{\prime },\beta _{0}^{\prime }\right) $ such that $\alpha _{0}^{\prime
}<\alpha _{0}$ and $\beta _{0}^{\prime }<\beta _{0}$. Therefore, since $\max
\left\{ 1,p\beta \right\} $ is nondecreasing in $\beta $ and $q^{*}\left(
\alpha ,\beta \right) $ is increasing in $\alpha $ and decreasing in $\beta $%
, it is convenient to choose $\beta _{0}=0$ and the best interval where one
can take $q_{1}$ is $1<q_{1}<q^{*}\left( \overline{\alpha },0\right) $ with $%
\overline{\alpha }:=\sup \left\{ \alpha _{0}:\esssup_{r\in \left(
0,R_{1}\right) }\frac{K\left( r\right) }{r^{\alpha _{0}}}<+\infty \right\} $
(we mean $q^{*}\left( +\infty ,0\right) =+\infty $).
\end{enumerate}
\end{rem}

For any $\alpha \in \mathbb{R}$, $\beta \leq 1$ and $\gamma \in \mathbb{R}$,
define 
\begin{equation}
q_{*}\left( \alpha ,\beta ,\gamma \right) :=p\frac{\alpha -\gamma \beta +N}{%
N-\gamma }\quad \text{and}\quad q_{**}\left( \alpha ,\beta ,\gamma \right)
:=
p\frac{p\alpha +\left( 1-p\beta \right) \gamma +p\left( N-1\right) }{%
p\left( N-1\right) -\gamma (p-1)}.  \label{q** :=}
\end{equation}
Of course $q_{*}$ and $q_{**}$ are undefined if $\gamma =N$ and $\gamma
= \frac{p}{p-1}\left( N-1\right) $, respectively.

The next Theorems \ref{THM2} and \ref{THM3} improve the results of Theorems 
\ref{THM0} and \ref{THM1} by exploiting further informations on the growth
of $V$ (see Remarks \ref{RMK: Hardy 1}.\ref{RMK: Hardy 1-improve} and \ref
{RMK: Hardy 2}.\ref{RMK: Hardy 2-improve}).

\begin{thm}
\label{THM2}
Let $1<p<N$ and let $V$, $K$ be as in $\left( \mathbf{V}\right) $, $\left( \mathbf{K}\right) $.
Assume that there exists $R_{2}>0$ such that $V\left( r\right) <+\infty $ for almost every $r>R_2$ and
\begin{equation}
\esssup_{r>R_{2}}\frac{K\left( r\right) }{r^{\alpha _{\infty
}}V\left( r\right) ^{\beta _{\infty }}}<+\infty \quad \text{for some }0\leq
\beta _{\infty }\leq 1\text{~and }\alpha _{\infty }\in \mathbb{R}
\label{hp all'inf}
\end{equation}
and 
\begin{equation}
\essinf_{r>R_{2}}r^{\gamma _{\infty }}V\left( r\right) >0\quad 
\text{for some }\gamma _{\infty } \leq p.  \label{stima all'inf}
\end{equation}
Then $\displaystyle \lim_{R\rightarrow +\infty }\mathcal{R}_{\infty }\left(
q_{2},R\right) =0$ for every $q_{2}\in \mathbb{R}$ such that 
\begin{equation}
q_{2}>\max \left\{ 1,p\beta _{\infty },q_{*},q_{**}\right\} ,  \label{th3}
\end{equation}
where $q_{*}=q_{*}\left( \alpha _{\infty },\beta _{\infty },\gamma _{\infty
}\right) $ and $q_{**}=q_{**}\left( \alpha _{\infty },\beta _{\infty
},\gamma _{\infty }\right) .$
\end{thm}

For future convenience, we define three functions $\alpha _{1}:=\alpha
_{1}\left( \beta ,\gamma \right) $, $\alpha _{2}:=\alpha _{2}\left( \beta
\right) $ and $\alpha _{3}:=\alpha _{3}\left( \beta ,\gamma \right) $ by
setting 
\begin{equation}
\alpha _{1}:=-\left( 1-\beta \right) \gamma ,\quad \alpha _{2}:=-\left(
1-\beta \right) N,\quad \alpha _{3}:=-\frac{(p-1) N+\left( 1-p\beta \right) \gamma 
}{p}.  \label{alpha_i :=}
\end{equation}
Then an explicit description of $\max \left\{ 1,p\beta ,q_{*},q_{**}\right\} 
$ is the following: for every $\left( \alpha ,\beta ,\gamma \right) \in \mathbb{%
R}\times \left( -\infty ,1\right] \times \left( -\infty ,N\right) $ we have 
\begin{equation}
\max \left\{ 1,p\beta ,q_{*},q_{**}\right\} =\left\{ 
\begin{array}{ll}
q_{**}\left( \alpha ,\beta ,\gamma \right) \quad & \text{if }\alpha \geq
\alpha _{1}\smallskip \\ 
q_{*}\left( \alpha ,\beta ,\gamma \right) & \text{if }\max \left\{ \alpha
_{2},\alpha _{3}\right\} \leq \alpha \leq \alpha _{1}\smallskip \\ 
\max \left\{ 1,p\beta \right\} & \text{if }\alpha \leq \max \left\{ \alpha
_{2},\alpha _{3}\right\}
\end{array}
\right. ,  \label{descrizioneThm2}
\end{equation}
where $\max \left\{ \alpha _{2},\alpha _{3}\right\} <\alpha _{1}$ for every $%
\beta <1$ and $\max \left\{ \alpha _{2},\alpha _{3}\right\} =\alpha _{1}=0$
if $\beta =1$.

\begin{rem}
\label{RMK: Hardy 1}\quad 

\begin{enumerate}
\item  \label{RMK: Hardy 1-B<0}The proof of Theorem \ref{THM2} does not
require $\beta _{\infty }\geq 0$, but this condition is not a restriction of
generality in stating the theorem. Indeed, under assumption (\ref{stima
all'inf}), if (\ref{hp all'inf}) holds with $\beta _{\infty }<0$, then it
also holds with $\alpha _{\infty }$ and $\beta _{\infty }$ replaced by $%
\alpha _{\infty }-\beta _{\infty }\gamma _{\infty }$ and $0$ respectively,
and this does not change the thesis (\ref{th3}), because $q_{*}\left( \alpha
_{\infty }-\beta _{\infty }\gamma _{\infty },0,\gamma _{\infty }\right)
=q_{*}\left( \alpha _{\infty },\beta _{\infty },\gamma _{\infty }\right) $
and $q_{**}\left( \alpha _{\infty }-\beta _{\infty }\gamma _{\infty
},0,\gamma _{\infty }\right) =q_{**}\left( \alpha _{\infty },\beta _{\infty
},\gamma _{\infty }\right) $.

\item  \label{RMK: Hardy 1-improve}

Denote $q^{*}=q^{*}\left( \alpha _{\infty
},\beta _{\infty }\right) $ for brevity. If $\gamma _{\infty }<p$, then one
has 
\[
\max \left\{ 1,p\beta _{\infty },q^{*}\right\} =\left\{ 
\begin{array}{ll}
\max \left\{ 1,p\beta _{\infty }\right\} =\max \left\{ 1,p\beta _{\infty
},q_{*},q_{**}\right\} \quad \smallskip  & \text{if }\alpha _{\infty }\leq
\alpha ^{*}\left( \beta _{\infty }\right)  \\ 
q^{*}>\max \left\{ 1,p\beta _{\infty },q_{*},q_{**}\right\}  & \text{if }%
\alpha _{\infty }>\alpha ^{*}\left( \beta _{\infty }\right) 
\end{array}
\right. ,
\]
so that, under assumption (\ref{stima all'inf}), Theorem \ref{THM2} improves
Theorem \ref{THM1}. Otherwise, if $\gamma _{\infty }=p$, we have $%
q_{*}=q_{**}=q^{*}$ and Theorems \ref{THM2} and \ref{THM1} give the same
result. This is not surprising, since, by Hardy inequality, the space $%
W$ coincides with $D^{1,p}(\mathbb{R}^{N})$ if $V\left( r\right) =r^{-p}
$ and thus, for $\gamma _{\infty }=p$, we cannot expect a better result than
the one of Theorem \ref{THM1}, which covers the case of $V\left( r\right)
\equiv 0$, i.e., of $D^{1,p}(\mathbb{R}^{N})$.

\item  \label{RMK: Hardy 1-best gamma}Description (\ref{descrizioneThm2})
shows that $q_{*}$ and $q_{**}$ are not relevant in inequality (\ref{th3}) if
$\alpha _{\infty }\leq \alpha _{2}\left( \beta _{\infty }\right) $. On the other hand, if 
$\alpha _{\infty }> \alpha _{2}\left( \beta _{\infty }\right) $, both $q_{*}$ and $q_{**}$ 
turn out to be increasing in $\gamma $ and
hence it is convenient to apply Theorem \ref{THM2} with the smallest $\gamma
_{\infty }$ for which (\ref{stima all'inf}) holds. This is consistent with
the fact that, if (\ref{stima all'inf}) holds with $\gamma _{\infty }$, then
it also holds with every $\gamma _{\infty }^{\prime }$ such that $\gamma
_{\infty }\leq \gamma _{\infty }^{\prime }\leq p$.
\end{enumerate}
\end{rem}

In order to state our last result, we introduce, by the following
definitions, an open region $\mathcal{A}_{\beta ,\gamma }$ of the $\alpha q$%
-plane, depending on $\beta\in[0,1]$ and $\gamma \geq p$. Recall the
definitions (\ref{q** :=}) of the functions $q_{*}=q_{*}\left( \alpha ,\beta
,\gamma \right) $ and $q_{**}=q_{**}\left( \alpha ,\beta ,\gamma \right) $.
We set 
\begin{equation}
\begin{array}{ll}
\mathcal{A}_{\beta ,\gamma }:=\left\{ \left( \alpha ,q\right) :\max \left\{
1,p\beta \right\} <q<\min \left\{ q_{*},q_{**}\right\} \right\} \quad
\smallskip & \text{if }p\leq\gamma <N, \\ 
\mathcal{A}_{\beta ,\gamma }:=\left\{ \left( \alpha ,q\right) :\max \left\{
1,p\beta \right\} <q<q_{**},\,\alpha >-\left( 1-\beta \right) N\right\}
\quad \smallskip & \text{if }\gamma =N, \\ 
\mathcal{A}_{\beta ,\gamma }:=\left\{ \left( \alpha ,q\right) :\max \left\{
1,p\beta ,q_{*}\right\} <q<q_{**}\right\} \smallskip & \text{if }N<\gamma
<\frac{p}{p-1}(N-1), \\ 
\mathcal{A}_{\beta ,\gamma }:=\left\{ \left( \alpha ,q\right) :\max \left\{
1,p\beta ,q_{*}\right\} <q,\,\alpha >-\left( 1-\beta \right) \gamma \right\}
\smallskip & \text{if }\gamma =\frac{p}{p-1}(N-1), \\ 
\mathcal{A}_{\beta ,\gamma }:=\left\{ \left( \alpha ,q\right) :\max \left\{
1,p\beta ,q_{*},q_{**}\right\} <q\right\} & \text{if }\gamma >\frac{p}{p-1}(N-1).
\end{array}
\label{A:=}
\end{equation}
Notice that $\frac{p}{p-1}(N-1)>N$ because $p<N$. For more clarity, $\mathcal{A}_{\beta ,\gamma }$ is sketched in the
following five pictures, according to the five cases above. Recall the
definitions (\ref{alpha_i :=}) of the functions $\alpha _{1}=\alpha
_{1}\left( \beta ,\gamma \right) $, $\alpha _{2}=\alpha _{2}\left( \beta
\right) $ and $\alpha _{3}=\alpha _{3}\left( \beta ,\gamma \right) $.\bigskip

\noindent
\begin{tabular}[t]{l}
\begin{tabular}{l}
\textbf{Fig.1}:\ \ $\mathcal{A}_{\beta ,\gamma }$ for %$\beta \leq 1$\\ and 
$p\leq \gamma <N$.\smallskip\\
%\\ 
$\bullet $ If $\gamma =p$, the two straight \\
lines above are the same.\smallskip\\
%\\ 
$\bullet $ If $\beta <1$ we have\\
$\max \left\{ \alpha _{2},\alpha _{3}\right\} <\alpha _{1}<0.$\\
If $\beta =1$ we have\\
$\alpha _{3}<\alpha _{2}=\alpha _{1}=0$\\
and $\mathcal{A}_{1,\gamma }$ reduces to the angle\\
$p<q<q_{**}$.
\end{tabular}
\begin{tabular}{l}
\includegraphics[width=3.7in]{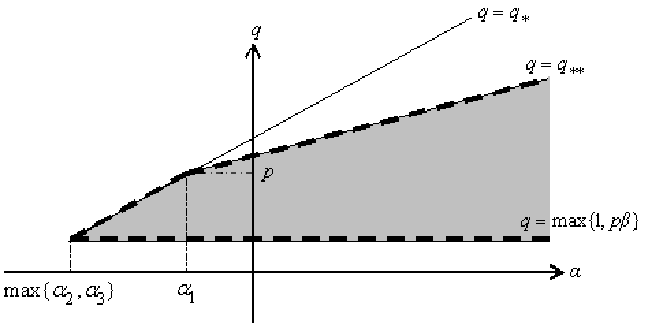}
\end{tabular}
\end{tabular}
\bigskip

\noindent
\begin{tabular}[t]{l}
\begin{tabular}{l}
\textbf{Fig.2}:\ \ $\mathcal{A}_{\beta ,\gamma }$ %%%for $\beta \leq 1$\\and 
$\gamma =N$.\smallskip\\
%\\ 
$\bullet $ If $\beta <1$ we have\\
$\alpha _{1}=\alpha _{2}=\alpha _{3}<0.$\\
If $\beta =1$ we have\\
$\alpha _{1}=\alpha _{2}=\alpha _{3}=0$\\
and $\mathcal{A}_{1,\gamma }$ reduces to the angle\\
$p<q<q_{**}$.
\end{tabular}
\begin{tabular}{l}
\includegraphics[width=3.7in]{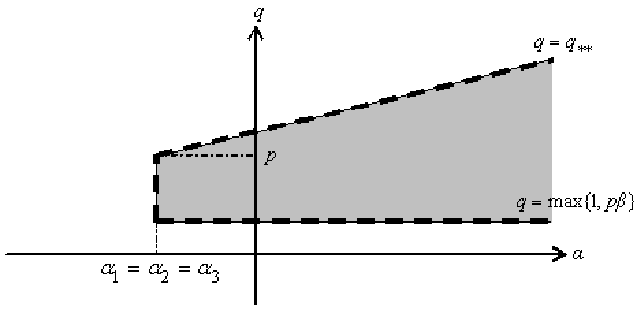}
\end{tabular}
\end{tabular}
\bigskip

\noindent
\begin{tabular}[t]{l}
\begin{tabular}{l}
\textbf{Fig.3}:\ \ $\mathcal{A}_{\beta ,\gamma }$ for \smallskip\\%%$\beta \leq 1$\\ and 
$N<\gamma<\frac{p}{p-1}(N-1)$.\smallskip\\
%\\ 
$\bullet $ If $\beta <1$ we have\\
$\alpha _{1}<\min \left\{ \alpha _{2},\alpha _{3}\right\} <0.$\\
If $\beta =1$ we have\\
$0=\alpha _{1}=\alpha_{2}<\alpha_{3}$\\
and $\mathcal{A}_{1,\gamma }$ reduces to the angle\\
$p<q<q_{**}$.
\end{tabular}
\begin{tabular}{l}
\includegraphics[width=3.7in]{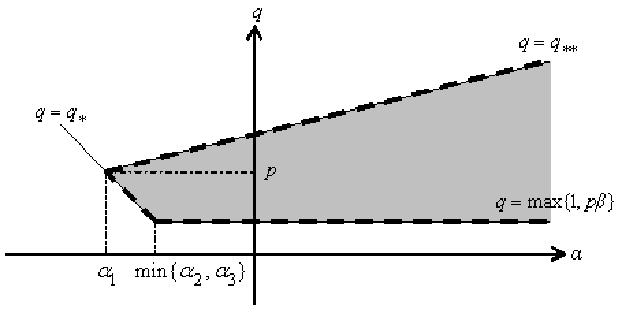}
\end{tabular}
\end{tabular}
\bigskip

\noindent
\begin{tabular}[t]{l}
\begin{tabular}{l}
\textbf{Fig.4}:\ \ $\mathcal{A}_{\beta ,\gamma }$ for %%%$\beta \leq 1$ \\ and 
$\gamma=\frac{p}{p-1}(N-1)$.\smallskip\\ 
%\\ 
$\bullet $ If $\beta <1$ we have \\ 
$\alpha _{1}<\min \left\{ \alpha _{2},\alpha _{3}\right\} <0.$ \\ 
If $\beta =1$ we have \\ 
$0=\alpha _{1}=\alpha_{2}<\alpha_{3}$ \\ 
and $\mathcal{A}_{1,\gamma }$ reduces to the angle\\ 
$\alpha >0,\,q>p$.
\end{tabular}
%\quad 
\begin{tabular}{l}
\includegraphics[width=3.4in]{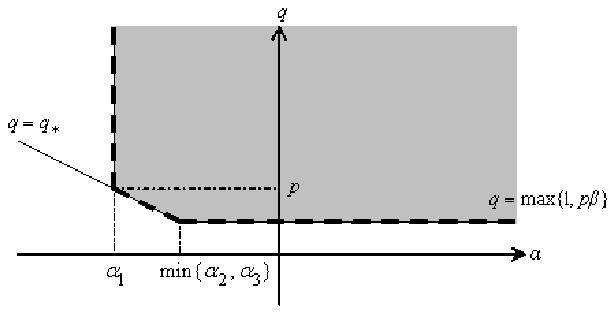}
\end{tabular}
\end{tabular}
\bigskip

\noindent
\begin{tabular}[t]{l}
\begin{tabular}{l}
\textbf{Fig.5}:\ \ $\mathcal{A}_{\beta ,\gamma }$ for %%%$\beta \leq 1$\\ and 
$\gamma>\frac{p}{p-1}(N-1)$.\smallskip\\
%\\ 
$\bullet $ If $\beta <1$ we have\\
$\alpha _{1}<\min \left\{ \alpha _{2},\alpha _{3}\right\} <0.$\\
If $\beta =1$ we have\\
$0=\alpha _{1}=\alpha_{2}<\alpha_{3}$\\
and $\mathcal{A}_{1,\gamma }$ reduces to the angle\\
$q>\max \left\{p,q_{**}\right\} $.
\end{tabular}
%\quad 
\begin{tabular}{l}
\includegraphics[width=3.4in]{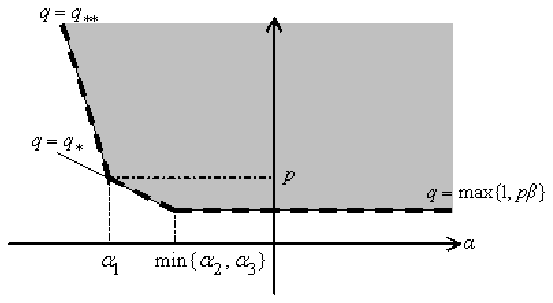}
\end{tabular}
\end{tabular}
\bigskip

\begin{thm}
\label{THM3}Let $N\geq 3$ and let $V$, $K$ be as in $\left( \mathbf{V}\right) $, $\left( \mathbf{K}\right)$.
%%%with $V\left( r\right) <+\infty $.
Assume that there exists $R_{1}>0$ such that 
$V\left( r\right) <+\infty $ almost everywhere in $(0,R_1)$ and
\begin{equation}
\esssup_{r\in \left( 0,R_{1}\right) }\frac{K\left( r\right) }{%
r^{\alpha _{0}}V\left( r\right) ^{\beta _{0}}}<+\infty \quad \text{for some }%
0\leq \beta _{0}\leq 1\text{~and }\alpha _{0}\in \mathbb{R}  \label{hp in 0}
\end{equation}
and 
\begin{equation}
\essinf_{r\in \left( 0,R_{1}\right) }r^{\gamma _{0}}V\left(
r\right) >0\quad \text{for some }\gamma _{0}\geq p.
  \label{stima in 0}
\end{equation}
Then $\displaystyle \lim_{R\rightarrow 0^{+}}\mathcal{R}_{0}\left(
q_{1},R\right) =0$ for every $q_{1}\in \mathbb{R}$ such that 
\begin{equation}
\left( \alpha _{0},q_{1}\right) \in \mathcal{A}_{\beta _{0},\gamma _{0}}.
\label{th4}
\end{equation}
\end{thm}

\begin{rem}
\label{RMK: Hardy 2}\quad 

\begin{enumerate}
\item  Condition (\ref{th4}) also asks for a lower bound on $\alpha _{0}$,
except for the case $\gamma _{0}>\frac{p}{p-1}(N-1)$, as it is clear from Figures 1-5.

\item  The proof of Theorem \ref{THM3} does not require $\beta _{0}\geq 0$,
but this is not a restriction of generality in stating the theorem (cf.
Remark \ref{RMK: Hardy 1}.\ref{RMK: Hardy 1-B<0}). Indeed, under assumption (%
\ref{stima in 0}), if (\ref{hp in 0}) holds with $\beta _{0}<0$, then it
also holds with $\alpha _{0}$ and $\beta _{0}$ replaced by $\alpha
_{0}-\beta _{0}\gamma _{0}$ and $0$ respectively, and one has that $\left(
\alpha _{0},q_{1}\right) \in \mathcal{A}_{\beta _{0},\gamma _{0}}$ if and
only if $\left( \alpha _{0}-\beta _{0}\gamma _{0},q_{1}\right) \in \mathcal{A%
}_{0,\gamma _{0}}$.

\item  \label{RMK: Hardy 2-improve}If (\ref{stima in 0}) holds with $\gamma
_{0}>p$, then Theorem \ref{THM3} improves Theorem \ref{THM0}. Otherwise, if $%
\gamma _{0}=p$, then one has $\max \left\{ \alpha _{2},\alpha _{3}\right\}
=\alpha ^{*}\left( \beta _{0}\right) $ and $\left( \alpha _{0},q_{1}\right)
\in \mathcal{A}_{\beta _{0},\gamma _{0}}$ is equivalent to $\max \left\{
1,p\beta _{0}\right\} <q_{1}<q^{*}\left( \alpha _{0},\beta _{0}\right) $,
i.e., Theorems \ref{THM3} and \ref{THM0} give the same result, which is
consistent with Hardy inequality (cf. Remark \ref{RMK: Hardy 1}.\ref{RMK:
Hardy 1-improve}).

\item  \label{RMK: Hardy 2-best gamma}Given $\beta \leq 1$, one can check
that $\mathcal{A}_{\beta ,\gamma _{1}}\subseteq \mathcal{A}_{\beta ,\gamma
_{2}}$ for every $p\leq \gamma _{1}<\gamma _{2}$, so that, in applying
Theorem \ref{THM3}, it is convenient to choose the largest $\gamma _{0}$ for
which (\ref{stima in 0}) holds. This is consistent with the fact that, if (%
\ref{stima in 0}) holds with $\gamma _{0}$, then it also holds with every $%
\gamma _{0}^{\prime }$ such that $p\leq \gamma _{0}^{\prime }\leq \gamma _{0}
$.
\end{enumerate}
\end{rem}

\section{Proof of Theorem \ref{THM(cpt)} \label{SEC:1}}

Assume $1<p<N$ and let $V$ and $K$ be as in $\left( \mathbf{V}\right) $
and $\left( \mathbf{K}\right) $.

Recall the definitions (\ref{spaces}) of the Banach spaces $W$ and $W_r$. 
Using the results of \cite[Lemma 1]{Su-Wang-Will-p}, fix two constants $S_{N,p}>0$ and 
$C_{N,p}>0$, only dependent on $N$ and $p$, such that 
\begin{equation}
\forall u\in W ,\quad 
\left\| u\right\|_{L^{p^{*}}(\mathbb{R}^{N})}\leq S_{N,p}\left\| u\right\|
  \label{Sobolev}
\end{equation}
and 
\begin{equation}
\forall u\in W_r  ,\quad 
\left|u\left( x\right) \right| \leq C_{N,p}\left\| u\right\| \frac{1}{\left|
x\right| ^{\frac{N-p}{p}}}\quad \text{almost everywhere on }\mathbb{R}^{N}.
\label{PointwiseEstimate}
\end{equation}

Recall from assumption $\left( \mathbf{K}\right) $ that $K\in L_{\mathrm{loc}%
}^{s}\left( \left( 0,+\infty \right) \right) $ for some $s>1$.
% $s>\frac{Np}{N(p-1)+p}\ (=\frac{p^{*}}{p^{*}-1})$.

\begin{lem}
\label{Lem(corone)}Let $R>r>0$ and $1<q<\infty $. 
Then there exist $\tilde{C}=\tilde{C}\left(N,p,r,R,q,s\right) >0$ and $l=l\left(p,q,s\right) >0$
such that $\forall u\in W_r$ one has 
\begin{equation}
\int_{B_{R}\setminus B_{r}}K\left( \left| x\right| \right) \left|u\right| ^q dx
\leq
\tilde{C}\left\| K\left( \left| \cdot\right| \right) \right\| _{L^{s}(B_{R}\setminus B_{r})}
\left\|u\right\|^{q-lp} \left(\int_{B_R\setminus B_r}\left|u\right|^{p}dx\right)^l.
\label{LEM:corone}
\end{equation}
Moreover, if 
$$s> \frac{p^*}{p^* -1} =  \frac{Np}{N(p-1)+p} \quad \left( p^* = \frac{Np}{N-p}\right)$$
in assumption $(\mathbf{K})$, then there exists $\tilde{C}_1=\tilde{C}_1\left(N,p,r,R,q,s\right)>0$ 
such that $\forall u\in W_r$ and $\forall h\in W$ one has 
\[
\frac{\int_{B_{R}\setminus B_{r}}K\left( \left| x\right| \right) \left|
u\right| ^{q-1}\left| h\right| dx}{\tilde{C}_1\left\| K\left( \left| \cdot
\right| \right) \right\| _{L^{s}(B_{R}\setminus B_{r})}}\leq \left\{ 
\begin{array}{ll}
\left( \int_{B_{R}\setminus B_{r}}\left| u\right| ^{p}dx\right) ^{\frac{q-1}{%
p}}\left\| h\right\| \medskip  & \text{if }q\leq \tilde{q} \\ 
\left( \int_{B_{R}\setminus B_{r}}\left| u\right| ^{p}dx\right) ^{\frac{%
\tilde{q}-1}{p}}\left\| u\right\| ^{q-\tilde{q}}\left\| h\right\| \quad
\medskip  & \text{if }q>\tilde{q}
\end{array}
\right. 
\]
where $\tilde{q}:=p\left( 1+\frac{1}{N}-\frac{1}{s}\right) $ (note that $s>%
\frac{Np}{N(p-1)+p}$ implies $\tilde{q}>1$).
\end{lem}

\proof%
Let $u\in W_r$ and fix $t\in(1,s)$ such that $t'q>p$ (where $t'=t/(t-1)$). 
Then, by H\"{o}lder inequality and (\ref{PointwiseEstimate}), we have 
\begin{eqnarray*}
&& \int_{B_{R}\setminus B_{r}}K\left( \left| x\right| \right) \left| u\right|^{q} dx \\
&\leq & \left( \int_{B_{R}\setminus B_{r}}K\left(\left| x\right| \right) ^{t}dx\right) ^{\frac{1}{t}}
\left( \int_{B_{R}\setminus B_{r}}\left|u\right| ^{t^{\prime }q}dx\right) ^{\frac{1}{t^{\prime }}} \\
&\leq & 
\left| B_{R}\setminus B_{r}\right| ^{\frac{1}{t}-\frac{1}{s}}
\left\| K\left( \left| \cdot \right| \right) \right\|_{L^{s}(B_{R}\setminus B_{r})}
\left( \int_{B_{R}\setminus B_{r}}\left|u\right| ^{t^{\prime }q-p}\left|u\right|^{p} dx\right) ^{\frac{1}{t^{\prime }}}
\\
&\leq &
\left| B_{R}\setminus B_{r}\right| ^{\frac{1}{t}-\frac{1}{s}}
\left\| K\left( \left| \cdot \right| \right) \right\|_{L^{s}(B_{R}\setminus B_{r})}
\left( \frac{C_{N,p}\left\|u\right\| }{r^{\frac{N-p}{p}}}\right) ^{q-p/t^{\prime}}
\left( \int_{B_{R}\setminus B_{r}}\left|u\right|^{p} dx\right) ^{\frac{1}{t^{\prime }}}.
\end{eqnarray*}
This proves (\ref{LEM:corone}). As to the second part of the lemma, let $u\in W_r$ and $h\in W$. 
For simplicity, we denote by $\sigma $ the H\"{o}lder-conjugate exponent of $%
p^{*}$, i.e., $\sigma := \frac{Np}{N(p-1)+p} %Np/\left( N(p-1)+p\right) 
$. By H\"{o}lder inequality
(note that $\frac{s}{\sigma }>1$), we have 
\begin{eqnarray*}
&&\int_{B_{R}\setminus B_{r}}K\left( \left| x\right| \right) \left| u\right|
^{q-1}\left| h\right| dx \\
&\leq & \left( \int_{B_{R}\setminus B_{r}}K\left(
\left| x\right| \right) ^{\sigma }\left| u\right| ^{\left( q-1\right) \sigma
}dx\right) ^{\frac{1}{\sigma }}\left( \int_{B_{R}\setminus B_{r}}\left|
h\right| ^{p^{*}}dx\right) ^{\frac{1}{p^{*}}} \\
&\leq & \left( \left( \int_{B_{R}\setminus B_{r}}K\left( \left| x\right|
\right) ^{s}dx\right) ^{\frac{\sigma }{s}}\left( \int_{B_{R}\setminus
B_{r}}\left| u\right| ^{\left( q-1\right) \sigma \left( \frac{s}{\sigma }%
\right) ^{\prime }}dx\right) ^{\frac{1}{\left( \frac{s}{\sigma }\right)
^{\prime }}}\right) ^{\frac{1}{\sigma }}S_{N,p}\left\| h\right\| \\
&\leq & S_{N,p}\left\| K\left( \left| \cdot \right| \right) \right\|
_{L^{s}(B_{R}\setminus B_{r})}\left\| h\right\| \left( \int_{B_{R}\setminus
B_{r}}\left| u\right| ^{p\frac{q-1}{\tilde{q}-1}}dx\right) ^{\frac{\tilde{q}%
-1}{p}},
\end{eqnarray*}
where we computed $\sigma \left( \frac{s}{\sigma }\right) ^{\prime }=\frac{%
pNs}{s N(p-1)+ps-pN}=\frac{p}{\tilde{q}-1}$. If $q < \tilde{q}$,
then we get 
\begin{eqnarray*}
&&\int_{B_{R}\setminus B_{r}}K\left( \left| x\right| \right) \left| u\right|
^{q-1}\left| h\right| dx \\
&\leq &S_{N,p}\left\| K\left( \left| \cdot \right|
\right) \right\| _{L^{s}(B_{R}\setminus B_{r})}\left\| h\right\| \left(
\left| B_{R}\setminus B_{r}\right| ^{1-\frac{q-1}{\tilde{q}-1}}\left(
\int_{B_{R}\setminus B_{r}}\left| u\right| ^{p}dx\right) ^{\frac{q-1}{\tilde{%
q}-1}}\right) ^{\frac{\tilde{q}-1}{p}} \\
&=&S_{N,p}\left\| K\left( \left| \cdot \right| \right) \right\|
_{L^{s}(B_{R}\setminus B_{r})}\left\| h\right\| \left| B_{R}\setminus
B_{r}\right| ^{\frac{\tilde{q}-q}{p}}\left( \int_{B_{R}\setminus
B_{r}}\left| u\right| ^{p}dx\right) ^{\frac{q-1}{p}}.
\end{eqnarray*}

If $q=\tilde{q}$ the thesis plainly follows.
Otherwise, if $q>\tilde{q}$, then by (\ref{PointwiseEstimate}) we obtain 
\begin{eqnarray*}
&&\int_{B_{R}\setminus B_{r}}K\left( \left| x\right| \right) \left| u\right|
^{q-1}\left| h\right| dx \\
&\leq &S_{N,p}\left\| K\left( \left| \cdot \right|
\right) \right\| _{L^{s}(B_{R}\setminus B_{r})}\left\| h\right\| \left(
\int_{B_{R}\setminus B_{r}}\left| u\right| ^{p\frac{q-1}{\tilde{q}-1}%
-p}\left| u\right| ^{p}dx\right) ^{\frac{\tilde{q}-1}{p}} \\
&\leq &S_{N,p}\left\| K\left( \left| \cdot \right| \right) \right\|
_{L^{s}(B_{R}\setminus B_{r})}\left\| h\right\| \left( \left( \frac{%
C_{N,p}\left\| u\right\| }{r^{\frac{N-p}{p}}}\right) ^{p\frac{q-\tilde{q}}{%
\tilde{q}-1}}\int_{B_{R}\setminus B_{r}}\left| u\right| ^{p}dx\right) ^{%
\frac{\tilde{q}-1}{p}} \\
&=&S_{N,p}\left\| K\left( \left| \cdot \right| \right) \right\|
_{L^{s}(B_{R}\setminus B_{r})}\left\| h\right\| \left( \frac{C_{N,p}\left\|
u\right\| }{r^{\frac{N-p}{p}}}\right) ^{q-\tilde{q}}\left(
\int_{B_{R}\setminus B_{r}}\left| u\right| ^{p}dx\right) ^{\frac{\tilde{q}-1%
}{p}}.
\end{eqnarray*}
This concludes the proof.%
\endproof%

For future reference, we recall here some results from \cite{BPR} concerning
the sum space 
\[
L_{K}^{p_{1}}+L_{K}^{p_{2}}:=L_{K}^{p_{1}}\left( \mathbb{R}^{N}\right)
+L_{K}^{p_{2}}\left( \mathbb{R}^{N}\right) :=\left\{ u_{1}+u_{2}:u_{1}\in
L_{K}^{p_{1}}\left( \mathbb{R}^{N}\right) ,\,u_{2}\in L_{K}^{p_{2}}\left( \mathbb{R%
}^{N}\right) \right\} , 
\]
where we assume $1<p_{1}\leq p_{2}<\infty $. Such a space can be
characterized as the set of the measurable mappings $u:\mathbb{R}%
^{N}\rightarrow \mathbb{R}$ for which there exists a measurable set $E\subseteq 
\mathbb{R}^{N}$ such that $u\in L_{K}^{p_{1}}\left( E\right) \cap
L_{K}^{p_{2}}\left( E^{c}\right) $ (of course $L_{K}^{p_{1}}\left( E\right)
:=L^{p_{1}}(E,K\left( \left| x\right| \right) dx)$, and so for $%
L_{K}^{p_{2}}\left( E^{c}\right) $). It is a Banach space with respect to
the norm 
\[
\left\| u\right\| _{L_{K}^{p_{1}}+L_{K}^{p_{2}}}:=\inf_{u_{1}+u_{2}=u}\max
\left\{ \left\| u_{1}\right\| _{L_{K}^{p_{1}}},\left\| u_{2}\right\|
_{L_{K}^{p_{2}}}\right\} 
\]
and the continuous embedding $L_{K}^{p}\hookrightarrow
L_{K}^{p_{1}}+L_{K}^{p_{2}}$ holds for all $p\in \left[ p_{1},p_{2}\right] $.

\begin{prop}[{\cite[Proposition 2.7]{BPR}}] 
\label{Prop(->0)}
Let $\left\{ u_{n}\right\}
\subseteq L_{K}^{p_{1}}+L_{K}^{p_{2}}$ be a sequence such that $\forall
\varepsilon >0$ there exist $n_{\varepsilon }>0$ and a sequence of
measurable sets $E_{\varepsilon ,n}\subseteq \mathbb{R}^{N}$ satisfying 
\begin{equation}
\forall n>n_{\varepsilon },\quad \int_{E_{\varepsilon ,n}}K\left( \left|
x\right| \right) \left| u_{n}\right| ^{p_{1}}dx+\int_{E_{\varepsilon
,n}^{c}}K\left( \left| x\right| \right) \left| u_{n}\right|
^{p_{2}}dx<\varepsilon .  \label{Prop(->0): cond}
\end{equation}
Then $u_{n}\rightarrow 0$ in $L_{K}^{p_{1}}+L_{K}^{p_{2}}$.
\end{prop}

\begin{prop}[{\cite[Propositions 2.17 and 2.14, Corollary 2.19]{BPR}}]
\label{Prop(L+L)}
Let $E\subseteq \mathbb{R}^{N}$ be a measurable set.

\begin{itemize}
\item[(i)]  If $\int_{E}K\left( \left| x\right| \right) dx<\infty $, then $%
L_{K}^{p_{1}}+L_{K}^{p_{2}}$ is continuously embedded into $%
L_{K}^{p_{1}}\left( E\right) $.

\item[(ii)]  Every $u\in (L_{K}^{p_{1}}+L_{K}^{p_{2}})\cap L^{\infty }\left(
E\right) $ satisfies 
\begin{equation}
\left\| u\right\| _{L_{K}^{p_{2}}\left( E\right) }^{p_{2}/p_{1}}\leq \left(
\left\| u\right\| _{L^{\infty }\left( E\right) }^{p_{2}/p_{1}-1}+\left\|
u\right\| _{L_{K}^{p_{2}}\left( E\right) }^{p_{2}/p_{1}-1}\right) \left\|
u\right\| _{L_{K}^{p_{1}}+L_{K}^{p_{2}}}.  \label{PropLL:1}
\end{equation}
If moreover $\left\| u\right\| _{L^{\infty }\left( E\right) }\leq 1$, then 
\begin{equation}
\left\| u\right\| _{L_{K}^{p_{2}}\left( E\right) }\leq 2\left\| u\right\|
_{L_{K}^{p_{1}}+L_{K}^{p_{2}}}+1.  \label{PropLL:2}
\end{equation}
\end{itemize}
\end{prop}

Recall the definitions (\ref{S_o :=})-(\ref{S_i :=}) of the functions $\mathcal{S}_{0}$ and $%
\mathcal{S}_{\infty }$.%\bigskip

%\noindent \textbf{Proof of Theorem \ref{THM(cpt)}.}\quad We prove each part of the theorem separately.\smallskip
\proof[Proof of Theorem \ref{THM(cpt)}]
We prove each part of the theorem separately.\smallskip

\noindent (i) By the monotonicity of $\mathcal{S}_{0}$ and $\mathcal{S}%
_{\infty }$, it is not restrictive to assume $R_{1}<R_{2}$ in hypothesis $%
\left( \mathcal{S}_{q_{1},q_{2}}^{\prime }\right) $. In order to prove the
continuous embedding, let $u\in W_r$, $u\neq 0$. Then we
have 
\begin{equation}
\int_{B_{R_{1}}}K\left( \left| x\right| \right) \left| u\right|
^{q_{1}}dx=\left\| u\right\| ^{q_{1}}\int_{B_{R_{1}}}K\left( \left| x\right|
\right) \frac{\left| u\right| ^{q_{1}}}{\left\| u\right\| ^{q_{1}}}dx\leq
\left\| u\right\| ^{q_{1}}\mathcal{S}_{0}\left( q_{1},R_{1}\right) 
\label{pf1}
\end{equation}
and, similarly, 
\begin{equation}
\int_{B_{R_{2}}^{c}}K\left( \left| x\right| \right) \left| u\right|
^{q_{2}}dx\leq \left\| u\right\| ^{q_{2}}\mathcal{S}_{\infty }\left(
q_{2},R_{2}\right) .  \label{pf2}
\end{equation}
We now use (\ref{LEM:corone}) of Lemma \ref{Lem(corone)}
and the continuous embedding 
$$W_r = D_{\mathrm{rad}}^{1,p}(\mathbb{R}^{N})\cap  L^p (\mathbb{R}^N,V(|x|)dx)\hookrightarrow L_{\mathrm{loc}}^{p}(\mathbb{R}^{N})$$ 
to deduce that
%, with $h=u$.  Treating separately the cases $q \leq \tilde{q}$ and $q> \tilde{q}$, 
%and recalling the continuous embedding 
%$W_r = D_{\mathrm{rad}}^{1,p}(\mathbb{R}^{N})\cap  L^p (\mathbb{R}^N,V(|x|)dx)\hookrightarrow L_{\mathrm{loc}}^{p}(\mathbb{R}^{N})$, 
%we deduce that 
there exists a constant $C_{1}>0$, independent from $u$, such that 
\begin{equation}
\int_{B_{R_{2}}\setminus B_{R_{1}}}K\left( \left| x\right| \right) \left|
u\right| ^{q_{1}}dx\leq C_{1}\left\| u\right\| ^{q_{1}}.  \label{pf3}
\end{equation}
Hence $u\in L_{K}^{q_{1}}(B_{R_{2}})\cap L_{K}^{q_{2}}(B_{R_{2}}^{c})$ and
thus $u\in L_{K}^{q_{1}}+L_{K}^{q_{2}}$. Moreover, if $u_{n}\rightarrow 0$
in $W_r$, then, using (\ref{pf1}), (\ref{pf2}) and (\ref
{pf3}), we get 
\[
\int_{B_{R_{2}}}K\left( \left| x\right| \right) \left| u_{n}\right|
^{q_{1}}dx+\int_{B_{R_{2}}^{c}}K\left( \left| x\right| \right) \left|
u_{n}\right| ^{q_{2}}dx=o\left( 1\right) _{n\rightarrow \infty },
\]
which means $u_{n}\rightarrow 0$ in $L_{K}^{q_{1}}+L_{K}^{q_{2}}$ by
Proposition \ref{Prop(->0)}. \emph{\smallskip }

\noindent (ii) Assume hypothesis $\left( \mathcal{S}_{q_{1},q_{2}}^{\prime
\prime }\right) $. Let $\varepsilon >0$ and let $u_{n}\rightharpoonup 0$ in $%
W_r$. Then $\left\{ \left\| u_{n}\right\| \right\} $ is
bounded and, arguing as for (\ref{pf1}) and (\ref{pf2}), we can take $%
r_{\varepsilon }>0$ and $R_{\varepsilon }>r_{\varepsilon }$ such that for
all $n$ one has 
\[
\int_{B_{r_{\varepsilon }}}K\left( \left| x\right| \right) \left|
u_{n}\right| ^{q_{1}}dx\leq \left\| u_{n}\right\| ^{q_{1}}\mathcal{S}%
_{0}\left( q_{1},r_{\varepsilon }\right) \leq \sup_{n}\left\| u_{n}\right\|
^{q_{1}}\mathcal{S}_{0}\left( q_{1},r_{\varepsilon }\right) <\frac{%
\varepsilon }{3}
\]
and 
\[
\int_{B_{R_{\varepsilon }}^{c}}K\left( \left| x\right| \right) \left|
u_{n}\right| ^{q_{2}}dx\leq \sup_{n}\left\| u_{n}\right\| ^{q_{2}}\mathcal{S}%
_{\infty }\left( q_{2},R_{\varepsilon }\right) <\frac{\varepsilon }{3}.
\]
Using (\ref{LEM:corone}) of Lemma \ref{Lem(corone)} and the boundedness of $\left\{ \left\|
u_{n}\right\| \right\} $ again, we infer that there exist two constants $%
C_{2},l>0$, independent from $n$, such that 
\[
\int_{B_{R_{\varepsilon }}\setminus B_{r_{\varepsilon }}}K\left( \left|
x\right| \right) \left| u_{n}\right| ^{q_{1}}dx\leq C_{2}\left(
\int_{B_{R_{\varepsilon }}\setminus B_{r_{\varepsilon }}}\left| u_{n}\right|
^{p}dx\right) ^{l},
\]
where 
\[
\int_{B_{R_{\varepsilon }}\setminus B_{r_{\varepsilon }}}\left| u_{n}\right|
^{p}dx\rightarrow 0\quad \text{as }n\rightarrow \infty \quad \text{(}%
\varepsilon ~\text{fixed)}
\]
thanks to the compactness of the embedding $D_{\mathrm{rad}}^{1,p}(\mathbb{R}%
^{N})\hookrightarrow L_{\mathrm{loc}}^{p}(\mathbb{R}^{N})$. Therefore we obtain 
\[
\int_{B_{R_{\varepsilon }}}K\left( \left| x\right| \right) \left|
u_{n}\right| ^{q_{1}}dx+\int_{B_{R_{\varepsilon }}^{c}}K\left( \left|
x\right| \right) \left| u_{n}\right| ^{q_{2}}dx<\varepsilon 
\]
for all $n$ sufficiently large, which means $u_{n}\rightarrow 0$ in $%
L_{K}^{q_{1}}+L_{K}^{q_{2}}$ (Proposition \ref{Prop(->0)}). This concludes
the proof of part (ii).\smallskip

\noindent (iii) First we observe that $K\left( \left| \cdot \right| \right)
\in L^{1}(B_{1})$ and assumption $\left( \mathbf{K}\right) $ imply $K\left(
\left| \cdot \right| \right) \in L_{\mathrm{loc}}^{1}(\mathbb{R}^{N})$. Assume $%
W_r\hookrightarrow L_{K}^{q_{1}}+L_{K}^{q_{2}}$ with $%
q_{1}\leq q_{2}$. Fix $R_{1}>0$. Then, by (i) of Proposition \ref{Prop(L+L)}%
, there exist two constants $c_{1},c_{2}>0$ such that $\forall u\in W_r$ we have 
\[
\int_{B_{R_{1}}}K\left( \left| x\right| \right) \left| u\right|
^{q_{1}}dx\leq c_{1}\left\| u\right\|
_{L_{K}^{q_{1}}+L_{K}^{q_{2}}}^{q_{1}}\leq c_{2}\left\| u\right\| ^{q_{1}}, 
\]
which implies $\mathcal{S}_{0}\left( q_{1},R_{1}\right) \leq c_{2}$. By (\ref
{PointwiseEstimate}), fix $R_{2}>0$ such that every $u\in W_r$ with 
$\left\| u\right\| =1$ satisfies $\left| u\left( x\right)\right| \leq 1$ 
almost everywhere on $B_{R_{2}}^{c}$. Then, by (\ref{PropLL:2}), we have 
\[
\int_{B_{R_{2}}^{c}}K\left( \left| x\right| \right) \left| u\right|
^{q_{2}}dx\leq \left( 2\left\| u\right\|
_{L_{K}^{q_{1}}+L_{K}^{q_{2}}}+1\right) ^{q_{2}}\leq \left( c_{3}\left\|
u\right\| +1\right) ^{q_{2}}=\left( c_{3}+1\right) ^{q_{2}} 
\]
for some constant $c_{3}>0$. This gives $\mathcal{S}_{\infty }\left(
q_{2},R_{2}\right) <\infty $ and thus $\left( \mathcal{S}_{q_{1},q_{2}}^{%
\prime }\right) $ holds (with $R_{1}>0$ arbitrary and $R_{2}$ large enough).
Now assume that the embedding $H_{V,\mathrm{r}}^{1}\hookrightarrow
L_{K}^{q_{1}}+L_{K}^{q_{2}}$ is compact and, by contradiction, that $%
\lim_{R\rightarrow 0^{+}}\mathcal{S}_{0}\left( q_{1},R\right) >\varepsilon
_{1}>0$ (the limit exists by monotonicity). Then for every $n\in \mathbb{N}%
\setminus \left\{ 0\right\} $ we have $\mathcal{S}_{0}\left(
q_{1},1/n\right) >\varepsilon _{1}$ and thus there exists $u_{n}\in W_r$ 
such that $\left\| u_{n}\right\| =1$ and 
\[
\int_{B_{1/n}}K\left( \left| x\right| \right) \left| u_{n}\right|
^{q_{1}}dx>\varepsilon _{1}. 
\]
Since $\left\{ u_{n}\right\} $ is bounded in $W_r$, by the
compactness assumption together with the continuous embedding $%
L_{K}^{q_{1}}+L_{K}^{q_{2}}\hookrightarrow L_{K}^{q_{1}}(B_{1})$ ((i) of
Proposition \ref{Prop(L+L)}), we get that there exists $u\in W_r$ 
such that, up to a subsequence, $u_{n}\rightarrow u$ in $L_{K}^{q_{1}}(B_{1})$. 
This implies 
\[
\int_{B_{1/n}}K\left( \left| x\right| \right) \left| u_{n}\right|
^{q_{1}}dx\rightarrow 0\quad \text{as }n\rightarrow \infty , 
\]
which is a contradiction. Similarly, if $\lim_{R\rightarrow +\infty }%
\mathcal{S}_{\infty }\left( q_{2},R\right) >\varepsilon _{2}>0$, then there
exists a sequence $\left\{ u_{n}\right\} \subset W_r$ such
that $\left\| u_{n}\right\| =1$ and 
\begin{equation}
\int_{B_{n}^{c}}K\left( \left| x\right| \right) \left| u_{n}\right|
^{q_{2}}dx>\varepsilon _{2}.  \label{>eps2}
\end{equation}
Moreover, we can assume that $\exists u\in W_r$ such that 
$u_{n}\rightharpoonup u$ in $W_r$, $u_{n}\rightarrow u$ in $L_{K}^{q_{1}}+L_{K}^{q_{2}}$ and 
\begin{equation}
\left\| u_{n}-u\right\| \leq \left\| u_{n}\right\| +\left\| u\right\| \leq
1+\liminf_{n\rightarrow \infty }\left\| u_{n}\right\| =2.  \label{un-u}
\end{equation}
Now, by (\ref{un-u}) and (\ref{PointwiseEstimate}), fix $R_{2}>0$ such that $%
\left| u_{n}\left( x\right) -u\left( x\right) \right| \leq 1$ almost
everywhere on $B_{R_{2}}^{c}$. Then $\left\{ u_{n}-u\right\} $ is bounded in 
$L_{K}^{q_{2}}(B_{R_{2}}^{c})$ by (\ref{PropLL:2}) and therefore (\ref
{PropLL:1}) gives 
\[
\int_{B_{R_{2}}^{c}}K\left( \left| x\right| \right) \left| u_{n}-u\right|
^{q_{2}}dx\leq c_{4}\left( \left\| u_{n}-u\right\|
_{L_{K}^{q_{1}}+L_{K}^{q_{2}}}\right) ^{q_{1}}\rightarrow 0\quad \text{as }%
n\rightarrow \infty 
\]
for some constant $c_{4}>0$. 
Since $u\in L_{K}^{q_{2}}(B_{R_{2}}^{c})$ by (\ref{PointwiseEstimate}) and (\ref{PropLL:1}), 
this implies 
\[
\int_{B_{n}^{c}}K\left( \left| x\right| \right) \left| u_{n}\right|
^{q_{2}}dx\rightarrow 0\quad \text{as }n\rightarrow \infty , 
\]
which contradicts (\ref{>eps2}). Hence we conclude $\lim_{R\rightarrow 0^{+}}%
\mathcal{S}_{0}\left( q_{1},R\right) =\lim_{R\rightarrow +\infty }\mathcal{S}%
_{\infty }\left( q_{2},R\right)$ $=0$, which completes the proof.%
%TCIMACRO{\TeXButton{End Proof}{\endproof}}
%BeginExpansion
\endproof%
%EndExpansion
%\bigskip

\section{Proof of Theorems \ref{THM0}\thinspace -\thinspace \ref{THM3} \label%
{SEC:2}}

Assume $1<p<N$ and let $V$ and $K$ be as in $\left( \mathbf{V}\right)$ and $\left( \mathbf{K}\right)$. %with $V\left( r\right) <+\infty $.
As in the previous section, we fix a constant $S_{N,p}>0$ such that (\ref{Sobolev}) holds.

\begin{lem}
\label{Lem(Omega)}Let $\Omega \subseteq \mathbb{R}^{N}$ be a nonempty
measurable set and assume that $V(r)<+\infty$ almost everywhere in $\Omega$ and
\[
\Lambda :=\esssup_{x\in \Omega }\frac{K\left( \left| x\right|
\right) }{\left| x\right| ^{\alpha }V\left( \left| x\right| \right) ^{\beta }%
}<+\infty \quad \text{for some }0\leq \beta \leq 1\text{~and }\alpha \in 
\mathbb{R}.
\]
Let $u\in W$ and assume that there exist $\nu \in \mathbb{R}$ and $m>0$
such that 
\[
\left| u\left( x\right) \right| \leq \frac{m}{\left| x\right| ^{\nu }}\quad 
\text{almost everywhere on }\Omega .
\]
Then $\forall h\in W$ and $\forall q>\max \left\{ 1,p\beta \right\} $%
, one has 

$\displaystyle\int_{\Omega }K\left( \left| x\right| \right) \left| u\right| ^{q-1}\left|
h\right| dx$
\[
%\int_{\Omega }K\left( \left| x\right| \right) \left| u\right| ^{q-1}\left|h\right| dx
\leq \left\{ 
\begin{array}{ll}
\Lambda m^{q-1}S_{N,p}^{1-p\beta }\left( \int_{\Omega }\left| x\right| ^{\frac{%
\alpha -\nu \left( q-1\right) }{N(p-1)+p\left( 1-p\beta \right) }pN}dx\right) ^{%
\frac{N(p-1)+p\left( 1-p\beta \right) }{pN}}\left\| h\right\| \quad \medskip  & 
\text{if }0\leq \beta \leq \frac{1}{p} \\ 
\Lambda m^{q-p\beta }\left( \int_{\Omega }\left| x\right| ^{\frac{\alpha
-\nu \left( q-p\beta \right) }{1-\beta }}dx\right) ^{1-\beta }\left\|
u\right\| ^{p\beta -1}\left\| h\right\| \medskip  & \text{if }\frac{1}{p}%
<\beta <1 \\ 
\Lambda m^{q-p}\left( \int_{\Omega }\left| x\right| ^{\frac{p}{p-1}(\alpha -\nu \left(
q-p\right)) }V\left( \left| x\right| \right) \left| u\right| ^{p}dx\right) ^{%
\frac{p-1}{p}}\left\| h\right\|  & \text{if }\beta =1.
\end{array}
\right. 
\]
\end{lem}

%TCIMACRO{\TeXButton{Proof}{\proof}}
%BeginExpansion
\begin{proof}
%\proof%
%EndExpansion
We distinguish several cases, where we will use H\"{o}lder inequality many
times, without explicitly noting it. \smallskip

\noindent \emph{Case }$\beta =0$\emph{. }%\smallskip

\noindent We have 
{\allowdisplaybreaks
\begin{eqnarray*}
\frac{1}{\Lambda }\int_{\Omega }K\left( \left| x\right| \right) \left|
u\right| ^{q-1}\left| h\right| dx 
&\leq & \int_{\Omega }\left| x\right|^{\alpha }\left| u\right| ^{q-1}\left| h\right| dx \\
&\leq & \left( \int_{\Omega
}\left( \left| x\right| ^{\alpha }\left| u\right| ^{q-1}\right) ^{\frac{pN}{%
N(p-1)+p}}dx\right) ^{\frac{N(p-1)+p}{pN}}\left( \int_{\Omega }\left| h\right|
^{p^{*}}dx\right) ^{\frac{1}{p^{*}}} \\
&\leq & m^{q-1}S_{N,p}\left( \int_{\Omega }\left| x\right| ^{\frac{\alpha -\nu
\left( q-1\right) }{N(p-1)+p}pN}dx\right) ^{\frac{N(p-1)+p}{pN}}\left\| h\right\| .
\end{eqnarray*}
}
%\smallskip

\noindent \emph{Case }$0<\beta <1/p$\emph{.}%\smallskip

\noindent One has $\frac{1}{\beta }>1$ and $\frac{1-\beta }{1-p\beta }%
p^{*}>1 $, with H\"{o}lder conjugate exponents $\left( \frac{1}{\beta }%
\right) ^{\prime }=\frac{1}{1-\beta }$ and $\left( \frac{1-\beta }{1-p\beta }%
p^{*}\right) ^{\prime }=\frac{pN\left( 1-\beta \right) }{N(p-1)+p\left( 1-p\beta
\right) }$. Then we get 
{\allowdisplaybreaks
\begin{eqnarray*}
&&\frac{1}{\Lambda }\int_{\Omega }K\left( \left| x\right| \right) \left|
u\right| ^{q-1}\left| h\right| dx 
\\
&\leq &\int_{\Omega }\left| x\right| ^{\alpha }V\left( \left| x\right|
\right) ^{\beta }\left| u\right| ^{q-1}\left| h\right| dx=\int_{\Omega
}\left| x\right| ^{\alpha }\left| u\right| ^{q-1}\left| h\right| ^{1-p\beta
}V\left( \left| x\right| \right) ^{\beta }\left| h\right| ^{p\beta }dx 
\\
&\leq &\left( \int_{\Omega }\left( \left| x\right| ^{\alpha }\left| u\right|
^{q-1}\left| h\right| ^{1-p\beta }\right) ^{\frac{1}{1-\beta }}dx\right)
^{1-\beta }\left( \int_{\Omega }V\left( \left| x\right| \right) \left|
h\right| ^{p}dx\right) ^{\beta } 
\\
%&\leq &\left( \int_{\Omega }\left( \left| x\right| ^{\alpha }\left| u\right|
%^{q-1}\left| h\right| ^{1-p\beta }\right) ^{\frac{1}{1-\beta }}dx\right)
%^{1-\beta }\left\| h\right\| ^{p\beta } \\
&\leq &\left( \left( \int_{\Omega }\left( \left| x\right| ^{\frac{\alpha }{%
1-\beta }}\left| u\right| ^{\frac{q-1}{1-\beta }}\right) ^{\left( \frac{%
1-\beta }{1-p\beta }p^{*}\right) ^{\prime }}dx\right) ^{\frac{1}{\left( 
\frac{1-\beta }{1-p\beta }p^{*}\right) ^{\prime }}}\left( \int_{\Omega
}\left| h\right| ^{p^{*}}dx\right) ^{\frac{1-p\beta }{\left( 1-\beta \right)
p^{*}}}\right) ^{1-\beta }\left\| h\right\| ^{p\beta } 
\\
&\leq &m^{q-1}\left( \left( \int_{\Omega }\left( \left| x\right| ^{\frac{%
\alpha }{1-\beta }-\nu \frac{q-1}{1-\beta }}\right) ^{\left( \frac{1-\beta }{%
1-p\beta }p^{*}\right) ^{\prime }}dx\right) ^{\frac{1}{\left( \frac{1-\beta 
}{1-\beta }p^{*}\right) ^{\prime }}}S_{N,p}^{\frac{1-p\beta }{1-\beta }%
}\left\| h\right\| ^{\frac{1-p\beta }{1-\beta }}\right) ^{1-\beta }\left\|
h\right\| ^{p\beta } 
\\
&=&m^{q-1}\left( \int_{\Omega }\left| x\right| ^{\frac{\alpha -\nu \left(
q-1\right) }{N(p-1)+p\left( 1-p\beta \right) }pN}dx\right) ^{\frac{N(p-1)+p\left(
1-p\beta \right) }{pN}}S_{N,p}^{1-p\beta }\left\| h\right\| .
\end{eqnarray*}
}
%\smallskip

\noindent \emph{Case }$\beta =\frac{1}{p}$\emph{.}%\smallskip

\noindent We have 
{\allowdisplaybreaks
\begin{eqnarray*}
\frac{1}{\Lambda }\int_{\Omega }K\left( \left| x\right| \right) \left|
u\right| ^{q-1}\left| h\right| dx &\leq &\int_{\Omega }\left| x\right|
^{\alpha }\left| u\right| ^{q-1}V\left( \left| x\right| \right) ^{\frac{1}{p}%
}\left| h\right| dx \\
&\leq &\left( \int_{\Omega }\left| x\right| ^{\alpha \frac{p}{p-1}}\left| u\right|
^{\left( q-1\right)\frac{p}{p-1} }dx\right) ^{\frac{p-1}{p}}\left( \int_{\Omega }V\left(
\left| x\right| \right) \left| h\right| ^{p}dx\right) ^{\frac{1}{p}} \\
&\leq &m^{q-1}\left( \int_{\Omega }\left| x\right| ^{(\alpha -\nu \left(
q-1\right) )\frac{p}{p-1} }dx\right) ^{\frac{p-1}{p}}\left\| h\right\| .
\end{eqnarray*}
}
%\smallskip

\noindent \emph{Case }$1/p<\beta <1$\emph{.}%\smallskip

\noindent One has $\frac{p-1}{p\beta -1}>1$, with H\"{o}lder conjugate
exponent $\left( \frac{p-1}{p\beta -1}\right) ^{\prime }=\frac{p-1}{p\left(
1-\beta \right) }$. Then 
{\allowdisplaybreaks
\begin{eqnarray*}
&&\frac{1}{\Lambda }\int_{\Omega }K\left( \left| x\right| \right) \left|
u\right| ^{q-1}\left| h\right| dx 
\\
&\leq &\int_{\Omega }\left| x\right|
^{\alpha }V\left( \left| x\right| \right) ^{\beta }\left| u\right|
^{q-1}\left| h\right| dx=\int_{\Omega }\left| x\right| ^{\alpha }V\left(
\left| x\right| \right) ^{\frac{p\beta -1}{p}}\left| u\right| ^{q-1}V\left(
\left| x\right| \right) ^{\frac{1}{p}}\left| h\right| dx 
\\
&\leq &\left( \int_{\Omega }\left| x\right| ^{\alpha \frac{p}{p-1}}V\left( \left|
x\right| \right) ^{\frac{p\beta -1}{p-1}}\left| u\right| ^{\left( q-1\right) \frac{p}{p-1}}dx\right)
^{\frac{p-1}{p}}\left( \int_{\Omega }V\left( \left| x\right| \right) \left|
h\right| ^{p}dx\right) ^{\frac{1}{p}} 
\\
&\leq &\left( \int_{\Omega }\left| x\right| ^{\alpha \frac{p}{p-1}}\left| u\right|
^{(q-1)\frac{p}{p-1}- p\frac{p\beta -1}{p-1} }V\left( \left| x\right| \right) ^{\frac{p\beta -1}{p-1}}\left| u\right| ^{p \frac{p\beta -1}{p-1} }dx\right) ^{\frac{p-1}{p}}\left\| h\right\| 
\\
&\leq &\left( \left( \int_{\Omega }\left| x\right| ^{\frac{\alpha }{1-\beta }%
}\left| u\right| ^{\frac{q-p\beta }{1-\beta }}dx\right) ^{\frac{p}{p-1}\left( 1-\beta
\right) }\left( \int_{\Omega }V\left( \left| x\right| \right) \left|
u\right| ^{p}dx\right) ^{\frac{p\beta -1}{p-1}}\right) ^{\frac{p-1}{p}}\left\| h\right\| 
\\
%&\leq &m^{q-p\beta }\left( \left( \int_{\Omega }\left| x\right| ^{\frac{%
%\alpha }{1-\beta }-\nu \frac{q-p\beta }{1-\beta }}dx\right) ^{\frac{p}{p-1}\left(
%1-\beta \right) }\left( \int_{\Omega }V\left( \left| x\right| \right) \left|
%u\right| ^{p}dx\right) ^{\frac{p\beta -1}{p-1}}\right) ^{\frac{p-1}{p}}\left\| h\right\| \\
&\leq &m^{q-p\beta }\left( \int_{\Omega }\left| x\right|^{\frac{\alpha }{1-\beta }-\nu \frac{q-p\beta }{1-\beta }}
%{\frac{\alpha -\nu(q-p\beta )}{1-\beta }}
dx\right) ^{1-\beta }\left( \int_{\Omega }V\left(
\left| x\right| \right) \left| u\right| ^{p}dx\right) ^{\frac{p\beta -1}{p}%
}\left\| h\right\| 
\\
&\leq &m^{q-p\beta }\left( \int_{\Omega }\left| x\right| ^{\frac{\alpha -\nu
(q-p\beta )}{1-\beta }}dx\right) ^{1-\beta }\left\| u\right\| ^{p\beta
-1}\left\| h\right\| .
\end{eqnarray*}
}
%\smallskip

\noindent \emph{Case }$\beta =1$\emph{.}%\smallskip

\noindent Assumption $q>\max \left\{ 1,p\beta \right\} $ means $q>p$ and
thus we have 
{\allowdisplaybreaks
\begin{eqnarray*}
&&\frac{1}{\Lambda }\int_{\Omega }K\left( \left| x\right| \right) \left|
u\right| ^{q-1}\left| h\right| dx 
\\
&\leq &\int_{\Omega }\left| x\right|
^{\alpha }V\left( \left| x\right| \right) \left| u\right| ^{q-1}\left|
h\right| dx = \int_{\Omega }\left| x\right| ^{\alpha }V\left( \left| x\right|
\right) ^{\frac{p-1}{p}}\left| u\right| ^{q-1}V\left( \left| x\right| \right)
^{\frac{1}{p}}\left| h\right| dx 
\\
&\leq &\left( \int_{\Omega }\left| x\right| ^{\alpha \frac{p}{p-1}}V\left( \left|
x\right| \right) \left| u\right| ^{\left( q-1\right)\frac{p}{p-1} }dx\right) ^{\frac{p-1}{p%
}}\left( \int_{\Omega }V\left( \left| x\right| \right) \left| h\right|
^{p}dx\right) ^{\frac{1}{p}} 
\\
&\leq &\left( \int_{\Omega }\left| x\right| ^{\alpha \frac{p}{p-1}}\left| u\right|
^{(q-1)\frac{p}{p-1}-p }V\left( \left| x\right| \right) \left| u\right|
^{p}dx\right) ^{\frac{p-1}{p}}\left\| h\right\| 
\\
&\leq &m^{q-p}\left( \int_{\Omega }\left| x\right| ^{\frac{p}{p-1}\left(\alpha -\nu (q-p) \right) }V\left( \left| x\right| \right) \left| u\right| ^{p}dx\right) ^{%
\frac{p-1}{p}}\left\| h\right\| .
\end{eqnarray*}
}
\end{proof}

As in the previous section, we fix a constant $C_{N,p}>0$ such that (\ref
{PointwiseEstimate}) holds. Recall the definitions (\ref{N_o})-(\ref{N_i})
of the functions $\mathcal{R}_{0}$ and $\mathcal{R}_{\infty }$.%\bigskip

%\noindent \textbf{Proof of Theorem \ref{THM0}.}\quad 
\begin{proof}[Proof of Theorem \ref{THM0}]
%\proof[Proof of Theorem \ref{THM0}]
Assume the hypotheses of the theorem and let $u\in W_r$ and $h\in W$ be
such that $\left\| u\right\| =\left\| h\right\| =1$. Let $0<R\leq R_{1}$. We
will denote by $C$ any positive constant which does not depend on $u$, $h$
and $R$.

By (\ref{PointwiseEstimate}) and the fact that 
\[
\esssup_{x\in B_{R}}\frac{K\left( \left| x\right| \right) }{\left|
x\right| ^{\alpha _{0}}V\left( \left| x\right| \right) ^{\beta _{0}}}\leq 
\esssup_{r\in \left( 0,R_{1}\right) }\frac{K\left( r\right) }{%
r^{\alpha _{0}}V\left( r\right) ^{\beta _{0}}}<+\infty , 
\]
we can apply Lemma \ref{Lem(Omega)} with $\Omega =B_{R}$, $\alpha =\alpha
_{0}$, $\beta =\beta _{0}$, $m=C_{N,p}\left\| u\right\| =C_{N,p}$ and $\nu =%
\frac{N-p}{p}$. If $0\leq \beta _{0}\leq 1/p$ we get 
\begin{eqnarray*}
\int_{B_{R}}K\left( \left| x\right| \right) \left| u\right| ^{q_{1}-1}\left|
h\right| dx &\leq &C\left( \int_{B_{R}}\left| x\right| ^{\frac{\alpha _{0}-%
\frac{N-p}{p}\left( q_{1}-1\right) }{N(p-1)+p\left( 1-p\beta _{0}\right) }%
pN}dx\right) ^{\frac{N(p-1)+p\left( 1-p\beta _{0}\right) }{pN}} 
\\
&\leq &C\left( \int_{0}^{R}r^{\frac{p\alpha _{0}-\left( N-p\right) \left(
q_{1}-1\right) }{N(p-1)+p\left( 1-p\beta _{0}\right) }N+N-1}dr\right) ^{\frac{%
N(p-1)+p\left( 1-p\beta _{0}\right) }{pN}} 
\\
&=&C\left( R^{\frac{p\alpha _{0}-p^2 \beta _{0}+pN-\left( N-p\right) q_{1}}{%
N(p-1)+p\left( 1-p\beta _{0}\right) }N}\right) ^{\frac{N(p-1)+p\left( 1-p\beta
_{0}\right) }{pN}},
\end{eqnarray*}
since 
\[
p\alpha _{0}- p^2 \beta _{0}+pN-\left( N-p\right) q_{1}=\left( N-p\right) \left(
q^{*}\left( \alpha _{0},\beta _{0}\right) -q_{1}\right) >0. 
\]
On the other hand, if $1/p<\beta _{0}<1$ we have 
\begin{eqnarray*}
\int_{B_{R}}K\left( \left| x\right| \right) \left| u\right| ^{q_{1}-1}\left|
h\right| dx 
&\leq & C\left( \int_{B_{R}}\left| x\right| ^{\frac{\alpha _{0}-%
\frac{N-p}{p}\left( q_{1}-p\beta _{0}\right) }{1-\beta _{0}}}dx\right)^{1-\beta _{0}}\\
&\leq & C\left( \int_{0}^{R}r^{\frac{\alpha _{0}-\frac{N-p}{p}%
\left( q_{1}-p\beta _{0}\right) }{1-\beta _{0}}+N-1}dr\right) ^{1-\beta _{0}}
\\
&=&C\left( R^{\frac{p\alpha _{0}-\left( N-p\right) \left( q_{1}-p\beta
_{0}\right) }{p\left( 1-\beta _{0}\right) }+N}\right) ^{1-\beta _{0}},
\end{eqnarray*}
since 
\[
\frac{p\alpha _{0}-\left( N-p\right) \left( q_{1}-p\beta _{0}\right) }{%
p\left( 1-\beta _{0}\right) }+N=\frac{N-p}{p\left( 1-\beta _{0}\right) }%
\left( q^{*}\left( \alpha _{0},\beta _{0}\right) -q_{1}\right) >0. 
\]
Finally, if $\beta _{0}=1$, we obtain 
\begin{eqnarray*}
\int_{B_{R}}K\left( \left| x\right| \right) \left| u\right| ^{q_{1}-1}\left|
h\right| dx &\leq &C\left( \int_{B_{R}}\left| x\right| ^{\frac{1}{p-1}\left( p \alpha _{0}-\left(
N-p\right) \left( q_{1}-p\right) \right) }V\left( \left| x\right| \right) \left|
u\right| ^{p}dx\right) ^{\frac{p-1}{p}} \\
&\leq &C\left( R^{\frac{1}{p-1}\left( p\alpha _{0}-\left( N-p\right) \left( q_{1}-p\right) \right)
}\int_{B_{R}}V\left( \left| x\right| \right) \left| u\right| ^{p}dx\right) ^{%
\frac{p-1}{p}}\\
&\leq & CR^{\frac{p\alpha _{0}-\left( N-p\right) \left(q_{1}-p\right) }{p}},
\end{eqnarray*}
since 
\[
p\alpha _{0}-\left( N-p\right) \left( q_{1}-p\right) =\left( N-p\right)
\left( q^{*}\left( \alpha _{0},1\right) -q_{1}\right) >0. 
\]
So, in any case, we deduce $\mathcal{R}_{0}\left( q_{1},R\right) \leq
CR^{\delta }$ for some $\delta =\delta \left( N,p,\alpha _{0},\beta
_{0},q_{1}\right) >0$ and this concludes the proof.
\end{proof}

%\noindent \textbf{Proof of Theorem \ref{THM1}.}\quad 
\proof[Proof of Theorem \ref{THM1}]
Assume the hypotheses of the theorem and let $u\in W_r$ and $h\in W$ be
such that $\left\| u\right\| =\left\| h\right\| =1$. Let $R\geq R_{2}$. We
will denote by $C$ any positive constant which does not depend on $u$, $h$
and $R$.

By (\ref{PointwiseEstimate}) and the fact that 
\[
\esssup_{x\in B_{R}^{c}}\frac{K\left( \left| x\right| \right) }{%
\left| x\right| ^{\alpha _{\infty }}V\left( \left| x\right| \right) ^{\beta
_{\infty }}}\leq \esssup_{r>R_{2}}\frac{K\left( r\right) }{%
r^{\alpha _{\infty }}V\left( r\right) ^{\beta _{\infty }}}<+\infty , 
\]
we can apply Lemma \ref{Lem(Omega)} with $\Omega =B_{R}^{c}$, $\alpha
=\alpha _{\infty }$, $\beta =\beta _{\infty }$, $m=C_{N,p}\left\| u\right\|
=C_{N,p}$ and $\nu =\frac{N-p}{p}$. If $0\leq \beta _{\infty }\leq 1/p$ we get 
\begin{eqnarray*}
\int_{B_{R}^{c}}K\left( \left| x\right| \right) \left| u\right|
^{q_{2}-1}\left| h\right| dx &\leq &C\left( \int_{B_{R}^{c}}\left| x\right|
^{\frac{\alpha _{\infty }-\frac{N-p}{p}\left( q_{2}-1\right) }{N(p-1)+p\left(
1-p\beta _{\infty }\right) }pN}dx\right) ^{\frac{N(p-1)+p\left( 1-p\beta _{\infty
}\right) }{pN}} \\
&\leq &C\left( \int_{R}^{+\infty }r^{\frac{p\alpha _{\infty }-\left(
N-p\right) \left( q_{2}-1\right) }{N(p-1)+p\left( 1-p\beta _{\infty }\right) }%
N+N-1}dr\right) ^{\frac{N(p-1)+p\left( 1-p\beta _{\infty }\right) }{pN}} \\
&=&C\left( R^{\frac{p\alpha _{\infty }-p^2 \beta _{\infty }+pN-\left(
N-p\right) q_{2}}{N(p-1)+p\left( 1-p\beta _{\infty }\right) }N}\right) ^{\frac{%
N(p-1)+p\left( 1-p\beta _{\infty }\right) }{pN}},
\end{eqnarray*}
since 
\[
p\alpha _{\infty }-p^2 \beta _{\infty }+pN-\left( N-p\right) q_{2}=\left(
N-p\right) \left( q^{*}\left( \alpha _{\infty },\beta _{\infty }\right)
-q_{2}\right) <0. 
\]
On the other hand, if $1/p<\beta _{\infty }<1$ we have 
\begin{eqnarray*}
\int_{B_{R}^{c}}K\left( \left| x\right| \right) \left| u\right|
^{q_{2}-1}\left| h\right| dx &\leq &C\left( \int_{B_{R}^{c}}\left| x\right|
^{\frac{\alpha _{\infty }-\frac{N-p}{p}\left( q_{2}-p\beta _{\infty }\right) 
}{1-\beta _{\infty }}}dx\right) ^{1-\beta _{\infty }} \\
&\leq &C\left( \int_{R}^{+\infty }r^{\frac{\alpha _{\infty }-\frac{N-p}{p}%
\left( q_{2}-p\beta _{\infty }\right) }{1-\beta _{\infty }}+N-1}dr\right)
^{1-\beta _{\infty }} \\
&=&C\left( R^{\frac{p\alpha _{\infty }-\left( N-p\right) \left( q_{2}-p\beta
_{\infty }\right) }{p\left( 1-\beta _{\infty }\right) }+N}\right) ^{1-\beta
_{\infty }},
\end{eqnarray*}
since 
\[
\frac{p\alpha _{\infty }-\left( N-p\right) \left( q_{2}-p\beta _{\infty
}\right) }{p\left( 1-\beta _{\infty }\right) }+N=\frac{N-p}{p\left( 1-\beta
_{_{\infty }}\right) }\left( q^{*}\left( \alpha _{_{\infty }},\beta
_{_{\infty }}\right) -q_{2}\right) <0. 
\]
Finally, if $\beta _{\infty }=1$, we obtain 
\begin{eqnarray*}
\int_{B_{R}^{c}}K\left( \left| x\right| \right) \left| u\right|
^{q_{2}-1}\left| h\right| dx &\leq &C\left( \int_{B_{R}^{c}}\left| x\right|
^{\frac{1}{p-1}(\left( p\alpha _{\infty }-\left( N-p\right) \left( q_{2}-p\right) \right) }V\left( \left|
x\right| \right) \left| u\right| ^{p}dx\right) ^{\frac{p-11}{p}} \\
&\leq &C\left( R^{\frac{1}{p-1} \left( p\alpha _{\infty }-\left( N-p\right) \left( q_{2}-p\right) \right)
}\int_{B_{R}^{c}}V\left( \left| x\right| \right) \left| u\right|
^{p}dx\right) ^{\frac{p-1}{p}}\\
&\leq & CR^{\frac{p\alpha _{\infty }-\left(
N-p\right) \left( q_{2}-p\right) }{p}},
\end{eqnarray*}
since 
\[
p\alpha _{\infty }-\left( N-p\right) \left( q_{2}-p\right) =\left(
N-p\right) \left( q^{*}\left( \alpha _{\infty },1\right) -q_{2}\right) <0. 
\]
So, in any case, we get $\mathcal{R}_{\infty }\left(
q_{2},R\right) \leq CR^{\delta }$ for some $\delta =\delta ( N,p,\alpha_{\infty },\beta _{\infty },q_{2}) <0$, 
which completes the proof.%
\endproof%

In proving Theorem \ref{THM2}, we will need the following lemma.

\begin{lem}
Assume that there exists $R_{2}>0$ such that 
$V(r)<+\infty$ for almost every $r>R_2$ and
\[
\lambda _{\infty }:=\essinf_{r>R_{2}}r^{\gamma _{\infty }}V\left(
r\right) >0\quad \text{for some }\gamma _{\infty }\leq p.
%%%< (N-1)\frac{p}{p-1}.
\]
Then there exists a constant $c_{\infty }>0$, only dependent on $N$ and $%
p$, such that 
\begin{equation}
\forall u\in W_r ,\quad \left|
u\left( x\right) \right| \leq c_{\infty }\lambda _{\infty }^{-\frac{p-1}{p^2}%
}\left\| u\right\| \left| x\right| ^{-\frac{p(N-1)-\gamma _{\infty } (p-1)}{p^2}%
}\quad \text{almost everywhere in }B_{R_{2}}^{c}.  \label{PointwiseInfty}
\end{equation}
\end{lem}

%%%%%%%%%%%%%%%%%%%%%%%%%%%%%%%%%%
%}
%%%%%%%%%%%%%%%%%% fine colore
%%%%%%%%%%%%%%%%%%%%%%%%%%%%%%%%%%

\proof%
The lemma is essentially proved in \cite[Lemma 4]{Su-Wang-Will-p}, but without making explicit the dependence of the constant on $\lambda_{\infty}$ and with slightly different assumptions on $V$: a global decay condition and the continuity on $(0,+\infty)$, which allows a density argument that is not so obvious in our hypotheses. In the form given here, instead, the lemma follows by adapting the proof of \cite[Proposition 28]{BGR_II}, where the result is proved for $p=2$. 
%
%The lemma follows by adapting the proof of \cite[Proposition 28]{BGR_II}, where the result is proved for $p=2$. 
%We point out that (\ref{PointwiseInfty}) is essentially proved also in \cite[Lemma 4]{Su-Wang-Will p}, but without making explicit the constants and using a density argument which relies on less general assumptions on $V$ (in particular, the authors assume that $V$ is globally continuous).
%
%.e., $V$ continuous and )with a global
%assumption on $V$ and without making explicit the constants, but, checking the proof, we see that the result actually holds in the form given here.
%with a global assumption on $V$ and without making explicit the constants, but, checking the proof, we see that the result actually holds in the form given here.
%
%The lemma is proved in \cite[Lemma 4]{Su-Wang-Will p} with a global assumption on $V$ and without making explicit the constants, but, checking the proof, we see that the result actually holds in the form given here.
\endproof%

%%%%%%%%%%%%%%%%%%%%%%%%%%%%%%%%%%
%{\color{blue}
%%%%%%%%%%%%%%%%%% inizio colore
%%%%%%%%%%%%%%%%%%%%%%%%%%%%%%%%%%

%\noindent \textbf{Proof of Theorem \ref{THM2}.}\quad 
\proof[Proof of Theorem \ref{THM2}]
Assume the hypotheses of the theorem and denote 
\[
\Lambda _{\infty }:=\esssup_{r>R_{2}}\frac{K\left( r\right) }{%
r^{\alpha _{\infty }}V\left( r\right) ^{\beta _{\infty }}}\quad \text{and}%
\quad \lambda _{\infty }:=\essinf_{r>R_{2}}r^{\gamma _{\infty
}}V\left( r\right) . 
\]
Let $u\in W_r $ and $h\in W$ be such that $\left\|
u\right\| =\left\| h\right\| =1$. Let $R\geq R_{2}$ and observe that $%
\forall \xi \geq 0$ one has 
\begin{equation}
\esssup_{r>R}\frac{K\left( r\right) }{r^{\alpha _{\infty }+\xi
\gamma _{\infty }}V\left( r\right) ^{\beta _{\infty }+\xi }}\leq 
\esssup_{r>R_{2}}\frac{K\left( r\right) }{r^{\alpha _{\infty }}V\left(
r\right) ^{\beta _{\infty }}\left( r^{\gamma _{\infty }}V\left( r\right)
\right) ^{\xi }}\leq \frac{\Lambda _{\infty }}{\lambda _{\infty }^{\xi }}%
<+\infty .  \label{stimaETA}
\end{equation}
We will denote by $C$ any positive constant which does not depend on $u$, $h$
or $R$ (such as $\Lambda _{\infty }/\lambda _{\infty }^{\xi }$ if $\xi $
does not depend on $u$, $h$ or $R$).

Denoting $\alpha _{1}=\alpha _{1}\left( \beta _{\infty },\gamma _{\infty
}\right) $, $\alpha _{2}=\alpha _{2}\left( \beta _{\infty }\right) $ and $%
\alpha _{3}=\alpha _{3}\left( \beta _{\infty },\gamma _{\infty }\right) $,
as defined in (\ref{alpha_i :=}), we will distinguish several cases,
according to the description (\ref{descrizioneThm2}). In each of such cases,
we will choose a suitable $\xi \geq 0$ and, thanks to (\ref{stimaETA}) and (%
\ref{PointwiseInfty}), we will apply Lemma \ref{Lem(Omega)} with $\Omega
=B_{R}^{c}$, $\alpha =\alpha _{\infty }+\xi \gamma _{\infty }$, $\beta
=\beta _{\infty }+\xi $ (whence $\Lambda $ will be given by the left hand
side of (\ref{stimaETA})), $m=c_{\infty }\lambda _{\infty }^{-\frac{p-1}{p^2}%
}\left\| u\right\| =c_{\infty }\lambda _{\infty }^{-\frac{p-1}{p^2}}$ and 
$\nu =\frac{p(N-1)-\gamma _{\infty }(p-1)}{p^2}$. 
Recall that we are assuming $\gamma _{\infty }\leq p<N$, so that $\nu>0$. %$\gamma _{\infty }< \frac{N-1}{p-1}p$). 
We will obtain that 
\[
\int_{B_{R}^{c}}K\left( \left| x\right| \right) \left| u\right|
^{q_{2}-1}\left| h\right| dx\leq CR^{\delta }
\]
for some $\delta <0$, not dependent on $R$, so that the result
follows.\medskip%\medskip 

\noindent \emph{Case }$\alpha _{\infty }\geq \alpha _{1}$.\smallskip

\noindent We take $\xi =1-\beta _{\infty }$ and apply Lemma \ref{Lem(Omega)}
with $\beta =\beta _{\infty }+\xi =1$ and $\alpha =\alpha _{\infty }+\xi
\gamma _{\infty }=\alpha _{\infty }+\left( 1-\beta _{\infty }\right) \gamma
_{\infty }$. We get 
\begin{eqnarray*}
\int_{B_{R}^{c}}K\left( \left| x\right| \right) \left| u\right|
^{q_{2}-1}\left| h\right| dx &\leq &C\left( \int_{B_{R}^{c}}\left| x\right|
^{\frac{p}{p-1}\left( \alpha -\nu \left( q_{2}-p\right) \right) }V\left( \left| x\right| \right)
\left| u\right| ^{p}dx\right) ^{\frac{p-1}{p}} \\
&\leq &C\left( R^{\frac{p}{p-1} \left( \alpha -\nu \left( q_{2}- p \right) \right) 
}\int_{B_{R}^{c}}V\left( \left| x\right| \right) \left| u\right|
^{p}dx\right) ^{\frac{p-1}{p}}\leq CR^{\alpha -\nu \left( q_{2}-p\right) },
\end{eqnarray*}
since 
%%%$\alpha -\nu \left( q_{2}-p\right) =\nu\left( q_{**}-q_{2}\right)<0$.\medskip
\begin{eqnarray*}
\alpha -\nu \left( q_{2}-p\right) &=& \alpha _{\infty }+\left( 1-\beta _{\infty
}\right) \gamma _{\infty }-\frac{p(N-1)-\gamma _{\infty }(p-1)}{p^2}\left(
q_{2}-p \right) \\
&=&\frac{p(N-1)-\gamma _{\infty }(p-1)}{p^2}\left( q_{**}-q_{2}\right)<0. 
\end{eqnarray*}

\noindent \emph{Case }$\max \left\{ \alpha _{2},\alpha _{3}\right\} <\alpha
_{\infty }<\alpha _{1}$.\smallskip

\noindent Take $\xi =\frac{\alpha _{\infty }+\left( 1-\beta _{\infty
}\right) N}{N-\gamma _{\infty }}>0$ and apply Lemma \ref{Lem(Omega)} with $%
\beta =\beta _{\infty }+\xi $ and $\alpha =\alpha _{\infty }+\xi \gamma
_{\infty }$. For doing this, observe that $\alpha _{3}<\alpha _{\infty
}<\alpha _{1}$ implies 
\[
\beta =\beta _{\infty }+\xi =\frac{\alpha _{\infty }-\gamma _{\infty }\beta
_{\infty }+N}{N-\gamma _{\infty }}\in \left( \frac{1}{p},1\right) . 
\]
We get 
\[
\int_{B_{R}^{c}}K\left( \left| x\right| \right) \left| u\right|
^{q_{2}-1}\left| h\right| dx\leq C\left( \int_{B_{R}^{c}}\left| x\right| ^{%
\frac{\alpha -\nu \left( q_{2}-p\beta \right) }{1-\beta }}dx\right)
^{1-\beta }\leq C\left( R^{\frac{\alpha -\nu \left( q_{2}-p\beta \right) }{%
1-\beta }+N}\right) ^{1-\beta }, 
\]
since 
\[
\frac{\alpha -\nu \left( q_{2}-p\beta \right) }{1-\beta }+N=\frac{\nu }{%
1-\beta }\left( p\frac{\alpha _{\infty }-\beta _{\infty }\gamma _{\infty }+N%
}{N-\gamma _{\infty }}-q_{2}\right) =\frac{\nu }{1-\beta }\left(
q_{*}-q_{2}\right) <0. 
\]
%\smallskip

\noindent \emph{Case }$\beta _{\infty }=1$\emph{\ and }$\alpha _{\infty
}\leq 0=\alpha _{2}\,\left( =\max \left\{ \alpha _{2},\alpha _{3}\right\}
\right) $.\smallskip

\noindent Take $\xi =0$ and apply Lemma \ref{Lem(Omega)} with $\beta =\beta
_{\infty }+\xi =1$ and $\alpha =\alpha _{\infty }+\xi \gamma _{\infty
}=\alpha _{\infty }$. We get 
\[
\int_{B_{R}^{c}}K\left( \left| x\right| \right) \left| u\right|
^{q_{2}-1}\left| h\right| dx\leq C\left( \int_{B_{R}^{c}}\left| x\right|
^{\frac{p}{p-1}\left( \alpha _{\infty }-\nu \left( q_{2}-p\right) \right) }V\left( \left| x\right|
\right) \left| u\right| ^{p}dx\right) ^{\frac{p-1}{p}}\leq CR^{\alpha _{\infty
}-\nu \left( q_{2}-p\right) }, 
\]
since $\alpha _{\infty }-\nu \left( q_{2}-p\right) \leq -\nu \left(
q_{2}-p\right) <0.$
\medskip

\noindent \emph{Case }$\frac{1}{p}<\beta _{\infty }<1$\emph{\ and }$\alpha
_{\infty }\leq \alpha _{2}\,\left( =\max \left\{ \alpha _{2},\alpha
_{3}\right\} \right) $.\smallskip

\noindent Take $\xi =0$ again and apply Lemma \ref{Lem(Omega)} with $\beta
=\beta _{\infty }\in \left( \frac{1}{p},1\right) $ and $\alpha =\alpha
_{\infty }$. We get 
\begin{eqnarray*}
\int_{B_{R}^{c}}K\left( \left| x\right| \right) \left| u\right|^{q_{2}-1}\left| h\right| dx 
&\leq & C\left( \int_{B_{R}^{c}}\left| x\right| ^{\frac{\alpha _{\infty }-\nu \left( q_{2}-p\beta _{\infty }\right) }{1-\beta
_{\infty }}}dx\right) ^{1-\beta _{\infty }}\\
&\leq & C\left( R^{\frac{\alpha_{\infty }-\nu \left( q_{2}-p\beta _{\infty }\right) }{1-\beta _{\infty }}%
+N}\right) ^{1-\beta _{\infty }}, 
\end{eqnarray*}
since 
\[
\frac{\alpha _{\infty }-\nu \left( q_{2}-p\beta _{\infty }\right) }{1-\beta
_{\infty }}+N=\frac{\alpha _{\infty }-\alpha _{2}-\nu \left( q_{2}-p\beta
_{\infty }\right) }{1-\beta _{\infty }}<0 
\]
%\smallskip

\noindent \emph{Case }$\beta _{\infty }\leq \frac{1}{p}$\emph{\ and }$\alpha
_{\infty }\leq \alpha _{3}\,\left( =\max \left\{ \alpha _{2},\alpha
_{3}\right\} \right) $.\smallskip

\noindent Take $\xi =\frac{1-p\beta _{\infty }}{p}\geq 0$, we can apply
Lemma \ref{Lem(Omega)} with $\beta =\beta _{\infty }+\xi =\frac{1}{p}$ and $%
\alpha =\alpha _{\infty }+\xi \gamma _{\infty }$. We get 
\[
\int_{B_{R}^{c}}K\left( \left| x\right| \right) \left| u\right|
^{q_{2}-1}\left| h\right| dx
\leq 
C\left( \int_{B_{R}^{c}}
\left| x\right|^{\frac{ \left( \alpha -\nu (q_2-1)\right)p}{p-1} }    dx\right)
^{\frac{p-1}{p}}
\leq
CR^{\alpha -\nu \left( q_{2}-1\right) +\frac{N(p-1)}{p}}, 
\]
since 
\[
\alpha -\nu \left( q_{2}-1\right) +\frac{N(p-1)}{p}=\alpha _{\infty }+\frac{%
1-p\beta _{\infty }}{p}\gamma _{\infty }+\frac{N(p-1)}{p}-\nu \left(
q_{2}-1\right) =\alpha _{\infty }-\alpha _{3}-\nu \left( q_{2}-1\right) <0. 
\]
\endproof%

%%%%%%%%%%%%%%%%%%%%%%%%%%%%%%%%%%
%{\color{blue}
%%%%%%%%%%%%%%%%%% inizio colore
%%%%%%%%%%%%%%%%%%%%%%%%%%%%%%%%%%

The proof of Theorem \ref{THM3} will be achieved by the following lemmas.

\begin{lem}
\label{LEM(pointwise0)}Assume that there exists $R>0$ such that 
$V(r)<+\infty$ almost everywhere in $(0,R)$ and
\[
\lambda \left( R\right) :=\essinf_{r\in \left( 0,R\right)
}r^{\gamma _{0}}V\left( r\right) >0\quad \text{for some }\gamma _{0}\geq p.
\]
Then there exists a constant $c_{0}>0$, only dependent on $N$ and $p$, 
such that $\forall u\in W_r$ %%% \cap D_{0}^{1,p}\left(B_{R}\right) $ 
one has 
\begin{equation}
\left| u\left( x\right) \right| \leq c_{0}\left( \left( \frac{1}{\lambda
\left( R\right)  }\right)^{\frac{p-1}{p}}+\frac{R^{\frac{\gamma _{0}-p}{p}}}{\lambda \left(
R\right) }\right) ^{\frac{1}{p}}\left\| u\right\| \left| x\right| ^{-\frac{%
p(N-1)-(p-1)\gamma _{0}}{p^2}}\quad \text{almost everywhere in }B_{R}.
\label{Pointwise0}
\end{equation}
\end{lem}

\proof%
The estimate (\ref{Pointwise0}) is essentially proved in \cite[Lemma 5]{Su-Wang-Will-p}, but without expliciting the dependence of the constants on $R$ and with slightly different assumptions on $V$ (a global decay condition and the continuity on $(0,+\infty)$) 
and $u$ (which is taken in $W_r \cap D_{0}^{1,p}\left(B_{R}\right)$). In the form given here, the lemma follows by adapting the proof of \cite[Proposition 29]{BGR_II}, where the result is proved for $p=2$. 
%
%The lemma is proved in \cite[Lemma 5]{Su-Wang-Will p} with a global assumption on $V$, and without expliciting the constants, but, checking the proof, we see that the result actually holds in the form given here.%
\endproof%

\begin{lem}
\label{Lem(Omega0)}Assume that there exists $R>0$ such that 
$V(r)<+\infty$ almost everywhere in $(0,R)$ and
\begin{equation}
\Lambda _{\alpha ,\beta }\left( R\right) :=\esssup_{r\in \left(
0,R\right) }\frac{K\left( r\right) }{r^{\alpha }V\left( r\right) ^{\beta }}%
<+\infty \quad \text{for some }\frac{1}{p}\leq \beta \leq 1\text{~and }%
\alpha \in \mathbb{R}  \label{Lem(Omega0): hp1}
\end{equation}
and 
\[
\lambda \left( R\right) :=\essinf_{r\in \left( 0,R\right)
}r^{\gamma _{0}}V\left( r\right) >0\quad \text{for some }\gamma _{0}>p.
\]
Assume also that there exists $q>p\beta $ such that 
\[
\left( p(N-1)-(p-1)\gamma _{0}\right) q < p^2 (\alpha +N)-p\left((p-1) \gamma _{0}+p\right)
\beta .
\]
Then $\forall u\in W_r$ %%%%%%%%%%  \cap D_{0}^{1,p}\left( B_{R}\right) $
and $\forall h\in W$ one has 
\[
\int_{B_{R}}K\left( \left| x\right| \right) \left| u\right| ^{q-1}\left|
h\right| dx\leq c_{0}^{q-p\beta }a\left( R\right) R^{\frac{p^2 (\alpha
+N)-p\left( (p-1) \gamma _{0}+p\right) \beta -\left( p(N-1)-(p-1)\gamma _{0}\right) q}{p^2}%
}\left\| u\right\| ^{q-1}\left\| h\right\| ,
\]
where $c_{0}$ is the constant of Lemma \ref{LEM(pointwise0)} and 
$a\left(R\right) :=\Lambda _{\alpha ,\beta }\left( R\right) \left(  \left( \frac{1}{%
\lambda \left( R\right) }\right)^{\frac{p-1}{p}}+\frac{R^{\frac{\gamma _{0}-p}{p}}}{\lambda \left(
R\right) }\right) ^{\frac{q-p\beta }{p}}$.
\end{lem}

\proof%
Let $u\in W_r$ %%%%%%%%%5%% \cap D_{0}^{1,p}\left( B_{R}\right) $ 
and $h\in W$. By assumption (\ref{Lem(Omega0): hp1}) and Lemma \ref
{LEM(pointwise0)}, we can apply Lemma \ref{Lem(Omega)} with $\Omega =B_{R}$, 
$\Lambda =\Lambda _{\alpha ,\beta }\left( R\right) $, $\nu =\frac{%
(N-1)p-(p-1)\gamma _{0}}{p^2}$ and 
\[
m=c_{0}\left( \left( \frac{1}{ \lambda \left( R\right) }\right)^{\frac{p-1}{p}}+\frac{R^{\frac{%
\gamma _{0}-p}{p}}}{\lambda \left( R\right) }\right) ^{\frac{1}{p}}\left\|
u\right\| . 
\]
If $\frac{1}{p}\leq \beta <1$, we get 
\begin{eqnarray*}
\int_{B_{R}}K\left( \left| x\right| \right) \left| u\right| ^{q-1}\left|
h\right| dx &\leq &\Lambda m^{q-p\beta }\left( \int_{\Omega }\left| x\right|
^{\frac{\alpha -\nu \left( q-p\beta \right) }{1-\beta }}dx\right) ^{1-\beta
}\left\| u\right\| ^{p\beta -1}\left\| h\right\| \\
&=&c_{0}^{q-p\beta }a\left( R\right) \left( \int_{B_{R}}\left| x\right| ^{%
\frac{p^2 \alpha -\left( p(N-1)-(p-1)\gamma _{0}\right) \left( q-p\beta \right) }{%
p^2 \left( 1-\beta \right) }}dx\right) ^{1-\beta }\left\| u\right\|
^{q-1}\left\| h\right\| \\
&\leq &c_{0}^{q-p\beta }a\left( R\right) \left( R^{\frac{p^2 \alpha -\left(
p(N-1)-(p-1)\gamma _{0}\right) \left( q-p\beta \right) }{p^2 \left( 1-\beta \right) }%
+N}\right) ^{1-\beta }\left\| u\right\| ^{q-1}\left\| h\right\| ,
\end{eqnarray*}
since 
\begin{eqnarray*}
&&\frac{p^2 \alpha -\left( p(N-1)-(p-1)\gamma _{0}\right) \left( q-p\beta \right) }{p^2 \left( 1-\beta \right) }+N\\
&=& \frac{p^2 (\alpha +N)-p\left((p-1) \gamma _{0}+p\right)
\beta -\left( p(N-1)-(p-1)\gamma _{0}\right) q}{p^2 \left( 1-\beta \right) }>0. 
\end{eqnarray*}
%\[
%\frac{p^2 \alpha -\left( p(N-1)-(p-1)\gamma _{0}\right) \left( q-p\beta \right) }{p^2 \left( 1-\beta \right) }+N=
%\]
%\[
%=\frac{p^2 (\alpha +N)-p\left((p-1) \gamma _{0}+p\right)\beta -\left( p(N-1)-(p-1)\gamma _{0}\right) q}{p^2 \left( 1-\beta \right) }>0. 
%\]
If instead we have $\beta =1$, we get 
\begin{eqnarray*}
&&\int_{B_{R}}K\left( \left| x\right| \right) \left| u\right| ^{q-1}\left|
h\right| dx \\
&\leq &\Lambda m^{q-p}\left( \int_{\Omega }\left| x\right|
^{\frac{p}{p-1}\left( \alpha -\nu \left( q-p\right)\right) }V\left( \left| x\right| \right) \left|
u\right| ^{p}dx\right) ^{\frac{p-1}{p}}\left\| h\right\| \\
&=&c_{0}^{q-p}a\left( R\right) \left( \int_{B_{R}}\left| x\right| ^{\frac{%
p^2 \alpha -\left( p(N-1)- (p-1)\gamma _{0}\right) \left( q-p\right) }{p(p-1)}}V\left(
\left| x\right| \right) \left| u\right| ^{p}dx\right) ^{\frac{p-1}{p}}\left\|
u\right\| ^{q-p}\left\| h\right\| \\
&\leq &c_{0}^{q-p}a\left( R\right) \left( R^{\frac{p^2 \alpha -\left(
p(N-1)-(p-1)\gamma _{0}\right) \left( q-p\right) }{p(p-1)}}\int_{B_{R}}V\left( \left|
x\right| \right) \left| u\right| ^{p}dx\right) ^{\frac{p-1}{p}}\left\|
u\right\| ^{q-p}\left\| h\right\| \\
&\leq &c_{0}^{q-p}a\left( R\right) R^{\frac{p^2 \alpha -\left( p(N-1)-(p-1)\gamma
_{0}\right) \left( q-p\right) }{p^2}}\left\| u\right\| ^{q-1}\left\| h\right\|,
\end{eqnarray*}
since 
\begin{eqnarray*}
&&p^2 \alpha -\left( p(N-1)-(p-1)\gamma _{0}\right) \left( q-p\right) \\
&=&p^2 (\alpha+N)-p\left((p-1) \gamma _{0}+p\right) -\left(p(N-1)-(p-1)\gamma _{0}\right) q>0. 
\end{eqnarray*}
%\[
%p^2 \alpha -\left( p(N-1)-(p-1)\gamma _{0}\right) \left( q-p\right) =
%\]
%\[
%=p^2 (\alpha+N)-p\left((p-1) \gamma _{0}+p\right) -\left( (N-1)p-(p-1)\gamma _{0}\right) q>0. 
%\]
\endproof%

\proof[Proof of Theorem \ref{THM3}]
Assume the hypotheses of the theorem and denote 
\[
\Lambda _{0}:=\esssup_{r\in \left( 0,R_{1}\right) }\frac{K\left(
r\right) }{r^{\alpha _{0}}V\left( r\right) ^{\beta _{0}}}\quad \text{and}%
\quad \lambda _{0}:=\essinf_{r\in \left( 0,R_{1}\right) }r^{\gamma
_{0}}V\left( r\right).
\]
If $\gamma _{0}=p$ the thesis of the theorem is true by Theorem \ref{THM0}
(see Remark \ref{RMK: Hardy 2}.\ref{RMK: Hardy 2-improve}), whence we can
assume $\gamma _{0}>p$ without restriction. 
We claim that for every $0<R\leq R_{1}$ there exists $b\left(
R\right) >0$ such that $b\left( R\right) \rightarrow 0$ as $R\rightarrow
0^{+}$ and 
\[
\int_{B_{R}}K\left( \left| x\right| \right) \left| u\right| ^{q_{1}-1}\left|
h\right| dx\leq b\left( R\right) \left\| u\right\| ^{q_{1}-1}\left\|
h\right\|,
\quad 
\forall u\in W_r, %%%%%%%%%% \cap D_{0}^{1,p}\left( B_{R}\right),
\ \forall h\in W,
%\text{for every }u\in W_r \cap D_{0}^{1,p}\left( B_{R}\right) \ \text{and }h\in W.
\]
which clearly gives the result. In order to prove this claim, let $0<R\leq R_{1}$.
Then one has
\begin{equation}
\lambda \left( R\right) :=\essinf_{r\in \left( 0,R\right)
}r^{\gamma _{0}}V\left( r\right) \geq \lambda _{0}>0  \label{stimaETA1}
\end{equation}
and for every $\xi \geq 0$ we have 
\begin{eqnarray}
\Lambda _{\alpha _{0}+\xi \gamma _{0},\beta _{0}+\xi }\left( R\right) &:=&
\esssup_{r\in \left( 0,R\right) }\frac{K\left( r\right) }{r^{\alpha
_{0}+\xi \gamma _{0}}V\left( r\right) ^{\beta _{0}+\xi }}\leq 
\esssup_{r\in \left( 0,R_{1}\right) }\frac{K\left( r\right) }{r^{\alpha
_{0}}V\left( r\right) ^{\beta _{0}}\left( r^{\gamma _{0}}V\left( r\right)
\right) ^{\xi }}\nonumber \\
&\leq &\frac{\Lambda _{0}}{\lambda _{0}^{\xi }}<+\infty .
\label{stimaETA0}
\end{eqnarray}
Denoting $\alpha _{1}=\alpha _{1}\left( \beta _{0},\gamma _{0}\right) $, $%
\alpha _{2}=\alpha _{2}\left( \beta _{0}\right) $ and $\alpha _{3}=\alpha
_{3}\left( \beta _{0},\gamma _{0}\right) $, as defined in (\ref{alpha_i :=}%
), we will now distinguish five cases, which reflect the five definitions (%
\ref{A:=}) of the set $\mathcal{A}_{\beta _{0},\gamma _{0}}$. For the sake
of clarity, some computations will be displaced in the Appendix.\medskip

\noindent \emph{Case }$p<\gamma _{0}<N$\emph{.}\smallskip

\noindent In this case, $\left( \alpha _{0},q_{1}\right) \in \mathcal{A}_{\beta _{0},\gamma _{0}}$ means 
\[
\begin{tabular}{l}
$\alpha _{0}>\max \left\{ \alpha _{2},\alpha _{3}\right\}$\quad and \smallskip
\\
%\lefteqn 
$\max \left\{ 1,p\beta _{0}\right\} <q_{1}<\displaystyle\min \left\{ p\frac{\alpha
_{0}-\beta _{0}\gamma _{0}+N}{N-\gamma _{0}},p\frac{p\alpha _{0}+\left(
1-p\beta _{0}\right) \gamma _{0}+p(N-1)}{p(N-1)-(p-1)\gamma _{0}}\right\}$\nonumber
\end{tabular}
\]
and these conditions ensure that we can fix $\xi \geq 0$, independent of $R$, 
in such a way that $\alpha =\alpha _{0}+\xi \gamma _{0}$
and $\beta =\beta _{0}+\xi $ satisfy 
\begin{equation}
\frac{1}{p}\leq \beta \leq 1\quad \text{and}\quad p\beta <q_{1}<\frac{%
p^2 (\alpha +N)-p\left((p-1) \gamma _{0}+p\right) \beta }{p(N-1)-(p-1)\gamma _{0}}
\label{Lem(zero): cond1}
\end{equation}
(see Appendix). Hence, by (\ref{stimaETA0}) and (\ref{stimaETA1}), we can
apply Lemma \ref{Lem(Omega0)} (with $q=q_{1}$), so that $\forall u\in W_r$ %%%%%% \cap D_{0}^{1,p}\left( B_{R}\right) $ 
and $\forall h\in W $ we get 
\[
\int_{B_{R}}K\left( \left| x\right| \right) \left| u\right| ^{q_{1}-1}\left|
h\right| dx\leq c_{0}^{q_{1}-p\beta }a\left( R\right) R^{\frac{p^2 (\alpha
+N)-p\left( (p-1) \gamma _{0}+p\right) \beta -\left( p(N-1)-(p-1)\gamma _{0}\right) q_{1}%
}{p^2}}\left\| u\right\| ^{q_{1}-1}\left\| h\right\| . 
\]
This gives the result, since $R^{p^2(\alpha +N)-p\left( (p-1) \gamma _{0}+p\right)
\beta -\left( p(N-1)- (p-1)\gamma _{0}\right) q_{1}}\rightarrow 0$ as $R\rightarrow
0^{+}$ and 
\[
a\left( R\right) =\Lambda _{\alpha _{0}+\xi \gamma _{0},\beta _{0}+\xi
}\left( R\right) \left( \left( \frac{1}{\lambda \left( R\right) }\right)^{\frac{p-1}{p}}+\frac{R^{%
\frac{\gamma _{0}-p}{p}}}{\lambda \left( R\right) }\right) ^{\frac{%
q_{1}-p\beta }{p}}\leq \frac{\Lambda _{0}}{\lambda _{0}^{\xi }}\left( \left( \frac{1%
}{\lambda _{0}}\right)^{\frac{p-1}{p}}+\frac{R_{1}^{\frac{\gamma _{0}-p}{p}}}{\lambda _{0}}%
\right) ^{\frac{q_{1}-p\beta }{p}}. 
\]
%\smallskip

\noindent \emph{Case }$\gamma _{0}=N$. \smallskip

\noindent In this case, $\left( \alpha _{0},q_{1}\right) \in \mathcal{A}%
_{\beta _{0},\gamma _{0}}$ means 
\[
\alpha _{0}>\alpha _{1}\,\left( =\alpha _{2}=\alpha _{3}\right) \quad \text{%
and}\quad \max \left\{ 1,p\beta _{0}\right\} <q_{1}<p\frac{p\alpha
_{0}+\left( 1-p\beta _{0}\right) \gamma _{0}+p(N-1)}{p(N-1)-(p-1)\gamma _{0}} 
\]
and these conditions still ensure that we can fix $\xi \geq 0$ in such a way
that $\alpha =\alpha _{0}+\xi \gamma _{0}$ and $\beta =\beta _{0}+\xi $
satisfy (\ref{Lem(zero): cond1}) (see Appendix), so that the result ensues
again by Lemma \ref{Lem(Omega0)}.\medskip

\noindent \emph{Case }$N<\gamma _{0}<\frac{p}{p-1}(N-1)$.\smallskip

\noindent In this case, $\left( \alpha _{0},q_{1}\right) \in \mathcal{A}%
_{\beta _{0},\gamma _{0}}$ means 
\[
\begin{tabular}{l}
$\alpha _{0}>\alpha _{1}$\quad and\smallskip\\
$\displaystyle\max \left\{ 1,p\beta _{0},p\frac{\alpha _{0}-\beta _{0}\gamma _{0}+N}{N-\gamma _{0}}\right\} <q_{1}
<\displaystyle p\frac{p\alpha _{0}+\left( 1-p\beta _{0}\right) \gamma _{0}+p(N-1)}{p(N-1)-(p-1)\gamma_{0}}$
\end{tabular}
\]
and the conclusion then follows as in the former cases (see
Appendix).\medskip

\noindent \emph{Case }$\gamma _{0}=\frac{p}{p-1}(N-1)$.\smallskip

\noindent In this case, $\left( \alpha _{0},q_{1}\right) \in \mathcal{A}%
_{\beta _{0},\gamma _{0}}$ means 
\[
\alpha _{0}>\alpha _{1}\quad \text{and}\quad \max \left\{ 1,p\beta _{0},p%
\frac{\alpha _{0}-\beta _{0}\gamma _{0}+N}{N-\gamma _{0}}\right\} <q_{1} 
\]
and these conditions ensure that we can fix $\xi \geq 0$ in such a way that $%
\alpha =\alpha _{0}+\xi \gamma _{0}$ and $\beta =\beta _{0}+\xi $ satisfy 
\[
\frac{1}{p}\leq \beta \leq 1,\quad q_{1}>p\beta \quad \text{and}\quad
0<p^2 \left( \alpha +N \right) - p\left( (p-1) \gamma _{0}+p\right) \beta 
\]
(see Appendix). The result then follows again from Lemma \ref{Lem(Omega0)}.\medskip

\noindent \emph{Case }$\gamma _{0}>\frac{p}{p-1} (N-1)$.\smallskip

\noindent In this case, $\left( \alpha _{0},q_{1}\right) \in \mathcal{A}%
_{\beta _{0},\gamma _{0}}$ means 
\[
q_{1}>\max \left\{ 1,p\beta _{0},p\frac{\alpha _{0}-\beta _{0}\gamma _{0}+N}{%
N-\gamma _{0}},p\frac{p\alpha _{0}+\left( 1-p\beta _{0}\right) \gamma
_{0}+p(N-1)}{p(N-1)- (p-1)\gamma _{0}}\right\} 
\]
and this condition ensures that we can fix $\xi \geq 0$ in such a way that $%
\alpha =\alpha _{0}+\xi \gamma _{0}$ and $\beta =\beta _{0}+\xi $ satisfy 
\[
\frac{1}{p}\leq \beta \leq 1\quad \text{and}\quad q_{1}>\max \left\{ p\beta
,p\frac{p(\alpha+N)-\left((p-1)\gamma _{0}+p\right) \beta }{p(N-1)-(p-1)\gamma_{0}}\right\} 
\]
(see Appendix). The result still follows from Lemma \ref{Lem(Omega0)}.
\endproof%

\section{Appendix}

This Appendix is devoted to complete the computations of the proof of Theorem \ref{THM3}.
%%%Lemma \ref{Lem(zero0)}. 
We still distinguish the same cases considered there.\medskip

\noindent \emph{Case }$p<\gamma _{0}<N$.\smallskip

\noindent In this case, $\left( \alpha _{0},q_{1}\right) \in \mathcal{A}%
_{\beta _{0},\gamma _{0}}$ means 
\[
\begin{tabular}{l}
$\alpha _{0}>\max \left\{ \alpha _{2},\alpha _{3}\right\}$ \quad and \smallskip\\
$\max \left\{ 1,p\beta _{0}\right\} <q_{1}<\displaystyle\min \left\{ p\frac{\alpha
_{0}-\beta _{0}\gamma _{0}+N}{N-\gamma _{0}},p\frac{p\alpha _{0}+\left(
1-p\beta _{0}\right) \gamma _{0}+p(N-1)}{p(N-1)-(p-1)\gamma _{0}}\right\}.$
\end{tabular}
\]
This ensures that we can find $\xi \geq 0$ such that 
\[
\frac{1}{p}\leq \beta _{0}+\xi \leq 1\quad \text{and}\quad p\left( \beta
_{0}+\xi \right) <q_{1}<\frac{p^2 \left( \alpha _{0}+\xi \gamma _{0}\right)
+p^2 N-p\left( (p-1) \gamma _{0}+p\right) \left( \beta _{0}+\xi \right) }{p(N-1)-(p-1)\gamma
_{0}}, 
\]
i.e., 
\[
\begin{tabular}{l}
$\frac{1}{p}-\beta _{0}\leq \xi \leq 1-\beta _{0}$ \quad and \smallskip\\
$p\beta_{0}+p\xi <q_{1}<p\frac{\gamma _{0}-p}{p(N-1)-(p-1)\gamma _{0}}\xi 
+\frac{p^2 \alpha_{0}+p^2 N-p\left( (p-1)\gamma _{0}+p\right) \beta _{0}}{p(N-1)-(p-1)\gamma _{0}}.$
\end{tabular}
\]
%\[
%\frac{1}{p}-\beta _{0}\leq \xi \leq 1-\beta _{0}\quad \text{and}\quad 
%p\beta_{0}+p\xi <q_{1}<p\frac{\gamma _{0}-p}{p(N-1)-(p-1)\gamma _{0}}\xi 
%+\frac{p^2 \alpha_{0}+p^2 N-p\left( (p-1)\gamma _{0}+p\right) \beta _{0}}{p(N-1)-(p-1)\gamma _{0}}. 
%\]
Indeed, this amounts to find $\xi $ such that 
\[
\left\{ 
\begin{array}{l}
\max \left\{ 0,\frac{1-p\beta _{0}}{p}\right\} \leq \xi \leq 1-\beta
_{0}\medskip \\ 
\xi <\frac{q_{1}-p\beta _{0}}{p}\medskip \\ 
q_{1}-\frac{p^2 \alpha _{0}+p^2 N-p\left( (p-1)\gamma _{0}+p\right) \beta _{0}}{%
p(N-1)-(p-1)\gamma _{0}}<p\frac{\gamma _{0}-p}{p(N-1)-(p-1)\gamma _{0}}\xi ,
\end{array}
\right. 
\]
which, since $\frac{\gamma _{0}-p}{p(N-1)-(p-1)\gamma _{0}}>0$, is equivalent to 
\[
\left\{ 
\begin{array}{l}
\max \left\{ 0,\frac{1-p\beta _{0}}{p}\right\} \leq \xi \leq 1-\beta
_{0}\medskip \\ 
q_{1}\frac{p(N-1)-(p-1)\gamma _{0}}{p\left( \gamma _{0}-p\right) }-\frac{p\alpha
_{0}+pN-\left( (p-1)\gamma _{0}+p\right) \beta _{0}}{\gamma _{0}-p}<\xi <\frac{%
q_{1}-p\beta _{0}}{p}.
\end{array}
\right. 
\]
Since $\frac{1}{p}-\beta _{0}\leq 1-\beta _{0}$ is obvious (recall that $p>1$) and $1-\beta
_{0}\geq 0$ holds by assumption, such a system has a solution $\xi $ if and
only if 
\[
\left\{ 
\begin{array}{l}
\frac{1-p\beta _{0}}{p}<\frac{q_{1}-p\beta _{0}}{p}\medskip \\ 
q_{1}\frac{p(N-1)-(p-1)\gamma _{0}}{p\left( \gamma _{0}-p\right) }-\frac{p\alpha
_{0}+pN-\left( (p-1) \gamma _{0}+p\right) \beta _{0}}{\gamma _{0}-p}<1-\beta
_{0}\medskip \\ 
q_{1}\frac{p(N-1)-(p-1)\gamma _{0}}{p\left( \gamma _{0}-p\right) }-\frac{p\alpha
_{0}+pN-\left( (p-1) \gamma _{0}+p\right) \beta _{0}}{\gamma _{0}-p}<\frac{%
q_{1}-p\beta _{0}}{p}\medskip \\ 
\frac{q_{1}-p\beta _{0}}{p}>0,
\end{array}
\right. 
\]
which is equivalent to 
\[
\left\{ 
\begin{array}{l}
1<q_{1}\medskip \\ 
q_{1}<p\frac{p\alpha _{0}+p(N-1)+\left( 1-p\beta _{0}\right) \gamma _{0}}{%
p(N-1)-(p-1)\gamma _{0}}\medskip \\ 
q_{1}<p\frac{\alpha _{0}+N-\gamma _{0}\beta _{0}}{N-\gamma _{0}}%
\medskip \\ 
q_{1}>p\beta _{0}.
\end{array}
\right. 
\]

\noindent \emph{Case }$\gamma _{0}=N$.\smallskip

\noindent In this case, $\left( \alpha _{0},q_{1}\right) \in \mathcal{A}%
_{\beta _{0},\gamma _{0}}$ means 
\[
\begin{tabular}{l}
$\alpha _{0}>\alpha _{1}\,\left( =\alpha _{2}=\alpha _{3}\right)$\quad and \smallskip\\
$\max \left\{ 1,p\beta _{0}\right\} <q_{1}<\displaystyle
p\frac{p\alpha_{0}+\left( 1-p\beta _{0}\right) \gamma _{0}+p(N-1)}{p(N-1)-(p-1)\gamma _{0}}=p\frac{%
p\alpha _{0}+(p+1)N-p\beta _{0}N-p}{N-p}   $ 
\end{tabular}
\]
and this ensures that we can find $\xi \geq 0$ such that 
\[
\frac{1}{p}\leq \beta _{0}+\xi \leq 1\quad \text{and}\quad p\left( \beta
_{0}+\xi \right) <q_{1}<\frac{p^2 \left( \alpha _{0}+\xi \gamma _{0}\right)
+p^2 N-p\left( (p-1) \gamma _{0}+p \right) \left( \beta _{0}+\xi \right) }{p(N-1)-(p-1) \gamma
_{0}}, 
\]
i.e., 
\[
\frac{1}{p}-\beta _{0}\leq \xi \leq 1-\beta _{0}\quad \text{and}\quad p\beta
_{0}+p\xi <q_{1}<p\xi +\frac{p^2 \alpha _{0}+p^2 N-p\left( N (p-1)+p\right) \beta _{0}}{%
N-p}. 
\]
Indeed, this amounts to find $\xi $ such that 
\[
\left\{ 
\begin{array}{l}
\max \left\{ 0,\frac{1-p\beta _{0}}{p}\right\} \leq \xi \leq 1-\beta
_{0}\medskip \\ 
\frac{q_{1}}{p}-\frac{p\alpha _{0}+pN-\left((p-1) N+p\right) \beta _{0}}{N-p}<\xi
<\frac{q_{1}-p\beta _{0}}{p},
\end{array}
\right. 
\]
which has a solution $\xi $ if and only if 
\[
\left\{ 
\begin{array}{l}
\frac{1-p\beta _{0}}{p}<\frac{q_{1}-p\beta _{0}}{p}\medskip \\ 
\frac{q_{1}}{p}-\frac{p\alpha _{0}+pN-\left( N(p-1)+p\right) \beta _{0}}{N-p}%
<1-\beta _{0}\medskip \\ 
\frac{q_{1}}{p}-\frac{p\alpha _{0}+pN-\left( N(p-1)+p\right) \beta _{0}}{N-p}<%
\frac{q_{1}-p\beta _{0}}{p}\medskip \\ 
0<\frac{q_{1}-\beta _{0}}{p}.
\end{array}
\right. 
\]
These conditions are equivalent to 
\[
\left\{ 
\begin{array}{l}
1<q_{1}\medskip \\ 
q_{1}<p\frac{p\alpha _0+(p+1)N-p\beta_0 N-p}{N-p}  \medskip \\ 
0<\frac{p\alpha _{0}+pN-\left( N(p-1)+p\right) \beta _{0}}{N-p}-\beta _{0}=p\frac{%
\alpha _{0}+N\left( 1-\beta _{0}\right) }{N-p}=p\frac{\alpha _{0}-\alpha _{1}%
}{N-p}\medskip \\ 
p\beta _{0}<q_{1}.
\end{array}
\right. 
\]
%\bigskip

\noindent \emph{Case }$N<\gamma _{0}<\frac{p}{p-1}(N-1)$.\smallskip 

\noindent In this case, $\left( \alpha _{0},q_{1}\right) \in \mathcal{A}%
_{\beta _{0},\gamma _{0}}$ means 
\[
\begin{tabular}{l}
$\alpha _{0}>\alpha _{1}$\quad and \smallskip\\
$\displaystyle\max \left\{ 1,p\beta _{0},p\frac{\alpha _{0}-\beta _{0}\gamma _{0}+N}{N-\gamma _{0}}\right\} <q_{1}<p%
\frac{p\alpha _{0}+\left( 1-p\beta _{0}\right) \gamma _{0}+p(N-1)}{p(N-1)-(p-1)\gamma_{0}}$
\end{tabular}
\]
and these conditions ensure that we can find $\xi \geq 0$ such that 
\[
\frac{1}{p}\leq \beta _{0}+\xi \leq 1\quad \text{and}\quad p\left( \beta
_{0}+\xi \right) <q_{1}<\frac{p^2 \left( \alpha _{0}+\xi \gamma _{0}\right)
+p^2 N-p\left( (p-1) \gamma _{0}+p\right) \left( \beta _{0}+\xi \right) }{(N-1)p-(p-1)\gamma
_{0}}, 
\]
i.e., 
\[
\begin{tabular}{l}
$\frac{1}{p}-\beta _{0}\leq \xi \leq 1-\beta _{0}$ \quad and \smallskip\\
$p\beta_{0}+p\xi <q_{1}<p\frac{\gamma _{0}-p}{p(N-1)-(p-1)\gamma _{0}}\xi 
+\frac{p^2 \alpha_{0}+p^2 N-p\left( (p-1)\gamma _{0}+p\right) \beta _{0}}{p(N-1)-(p-1)\gamma _{0}}.$
\end{tabular}
%\frac{1}{p}-\beta _{0}\leq \xi \leq 1-\beta _{0}\quad \text{and}\quad 
%p\beta_{0}+p\xi <q_{1}<p\frac{\gamma _{0}-p}{p(N-1)-(p-1)\gamma _{0}}\xi 
%+\frac{p^2 \alpha_{0}+p^2 N-p\left( (p-1)\gamma _{0}+p\right) \beta _{0}}{p(N-1)-(p-1)\gamma _{0}}. 
\]
Indeed, this is equivalent to find $\xi $ such that 
\[
\left\{ 
\begin{array}{l}
\max \left\{ 0,\frac{1-p\beta _{0}}{p}\right\} \leq \xi \leq 1-\beta
_{0}\medskip \\ 
\xi <\frac{q_{1}-p\beta _{0}}{p}\medskip \\ 
q_{1}-\frac{p^2 \alpha _{0}+p^2 N-p\left( (p-1)\gamma _{0}+p\right) \beta _{0}}{%
p(N-1)-(p-1)\gamma _{0}}<p\frac{\gamma _{0}-p}{p(N-1)-(p-1)\gamma _{0}}\xi ,
\end{array}
\right. 
\]
which, since $\frac{\gamma _{0}-p}{p(N-1)-(p-1)\gamma _{0}}>0$, amounts to 
\[
\left\{ 
\begin{array}{l}
\max \left\{ 0,\frac{1-p\beta _{0}}{p}\right\} \leq \xi \leq 1-\beta
_{0}\medskip \\ 
\xi <\frac{q_{1}-p\beta _{0}}{p}\medskip \\ 
q_{1}\frac{p(N-1)-(p-1)\gamma _{0}}{p\left( \gamma _{0}-p\right) }-\frac{p\alpha
_{0}+pN-\left( (p-1)\gamma _{0}+p\right) \beta _{0}}{\gamma _{0}-p} <\xi .
\end{array}
\right. 
\]
Such a system has a solution $\xi $ if and only if 
\[
\left\{ 
\begin{array}{l}
0<\frac{q_{1}-p\beta _{0}}{p}\medskip \\ 
\frac{1-p\beta _{0}}{p}<\frac{q_{1}-p\beta _{0}}{p}\medskip \\ 
q_{1}\frac{p(N-1)-(p-1)\gamma _{0}}{p\left( \gamma _{0}-p\right) }-\frac{p\alpha
_{0}+pN-\left( (p-1)\gamma _{0}+p\right) \beta _{0}}{\gamma _{0}-p}  <1-\beta
_{0}\medskip \\ 
q_{1}\frac{p(N-1)-(p-1)\gamma _{0}}{p\left( \gamma _{0}-p\right) }-\frac{p\alpha
_{0}+pN-\left( (p-1)\gamma _{0}+p\right) \beta _{0}}{\gamma _{0}-p} <\frac{%
q_{1}-p\beta _{0}}{p},
\end{array}
\right. 
\]
which is equivalent to 
\[
\left\{ 
\begin{array}{l}
q_{1}>p\beta _{0}\medskip \\ 
q_{1}>1\medskip \\ 
q_{1}<p\frac{p\alpha _{0}+pN-p\beta _{0}\gamma _{0}+\gamma _{0}-p}{%
p(N-1)-(p-1)\gamma _{0}}\medskip \\ 
-p\frac{\alpha _{0}+N-\beta _{0}\gamma _{0}}{\gamma _{0}-p}<q_{1}\frac{%
\gamma _{0}-N}{\gamma _{0}-p}.
\end{array}
\right. 
\]
%\bigskip

\noindent \emph{Case }$\gamma _{0}=\frac{p}{p-1}(N-1)$\emph{.\smallskip }

\noindent In this case, $\left( \alpha _{0},q_{1}\right) \in \mathcal{A}%
_{\beta _{0},\gamma _{0}}$ means 
\[
\alpha _{0}>\alpha _{1}\quad \text{and}\quad q_{1}>\max \left\{ 1,p\beta
_{0},-p\frac{(\alpha _{0}+N)(p-1)-p(N-1)\beta_{0}}{N-p}\right\} . 
\]
This ensures that we can find $\xi \geq 0$ such that 
\[
\frac{1}{p}\leq \beta _{0}+\xi \leq 1,\quad q_{1}>p\left( \beta _{0}+\xi
\right) \quad \text{and}\quad p\left( \alpha _{0}+\xi \gamma _{0}\right)
+pN-\left( (p-1)\gamma _{0}+p\right) \left( \beta _{0}+\xi \right) >0, 
\]
i.e., 
\[
\frac{1}{p}-\beta _{0}\leq \xi \leq 1-\beta _{0},\quad q_{1}>p\beta
_{0}+p\xi \quad \text{and}\quad \alpha _{0}+N\left( 1-\beta _{0}\right)
+\frac{N-p}{p-1} \xi >0. 
\]
Indeed, this amounts to find $\xi $ such that 
\[
\left\{ 
\begin{array}{l}
\max \left\{ 0,\frac{1-p\beta _{0}}{p}\right\} \leq \xi \leq 1-\beta
_{0}\medskip \\ 
-\frac{(p-1) \left(\alpha _{0}+N\left( 1-\beta _{0}\right) \right)}{N-p}<\xi <\frac{%
q_{1}-p\beta _{0}}{p}
\end{array}
\right. 
\]
and such a system has a solution $\xi $ if and only if 
\[
\left\{ 
\begin{array}{l}
\frac{1-p\beta _{0}}{p}<\frac{q_{1}-p\beta _{0}}{p}\medskip \\ 
-\frac{(p-1) \left( \alpha _{0}+N\left( 1-\beta _{0}\right) \right)  }{N-p} <1-\beta _{0}\medskip
\\ 
-\frac{(p-1) \left( \alpha _{0}+N\left( 1-\beta _{0}\right) \right)  }{N-p}  <\frac{q_{1}-p\beta _{0}%
}{p}\medskip \\ 
0<\frac{q_{1}-p\beta _{0}}{p},
\end{array}
\right. 
\]
which means 
\[
\left\{ 
\begin{array}{l}
1<q_{1}\medskip \\ 
\alpha _{0}>-(N-1) \frac{p}{p-1} \left( 1-\beta _{0}\right) =\alpha
_{1}\medskip \\ 
\frac{q_{1}}{p}>\beta _{0}-\frac{(p-1) \left( \alpha _{0}+N\left( 1-\beta _{0}\right) \right) }{N-p}
=-\frac{ (p-1)n\left( \alpha _{0}+N \right) -p\left( N-1\right) \beta _{0}}{N-p}\medskip \\ 
p\beta _{0}<q_{1}.
\end{array}
\right. 
\]
%\bigskip

\noindent \emph{Case }$\gamma _{0}>\frac{p}{p-1} (N-1)$.\smallskip

\noindent In this case, $\left( \alpha _{0},q_{1}\right) \in \mathcal{A}%
_{\beta _{0},\gamma _{0}}$ means 
\[
q_{1}>\max \left\{ 1,p\beta _{0},p\frac{\alpha _{0}-\beta _{0}\gamma _{0}+N}{%
N-\gamma _{0}},p\frac{p\alpha _{0}+\left( 1-p\beta _{0}\right) \gamma
_{0}+pN-p}{p(N-1)-(p-1)\gamma _{0}}\right\} 
\]
and this condition ensures that we can find $\xi \geq 0$ such that 
\[
\frac{1}{p}\leq \beta _{0}+\xi \leq 1\quad \text{and}\quad q_{1}>p\max
\left\{ \beta _{0}+\xi ,\frac{p\left( \alpha _{0}+\xi \gamma _{0}\right)
+pN-\left( (p-1)\gamma _{0}+p\right) \left( \beta _{0}+\xi \right) }{p(N-1)-(p-1)\gamma
_{0}}\right\} , 
\]
which amounts to find $\xi $ such that 
\begin{equation}
\left\{ 
\begin{array}{l}
\max \left\{ 0,\frac{1-p\beta _{0}}{p}\right\} \leq \xi \leq 1-\beta
_{0}\medskip \\ 
\frac{q_{1}}{p}>\max \left\{ \beta _{0}+\xi ,\frac{\gamma _{0}-p}{%
p(N-1)-(p-1)\gamma _{0}}\xi +\frac{p\alpha _{0}+pN-\left( (p-1)\gamma _{0}+p\right)
\beta _{0}}{p(N-1)-(p-1)\gamma _{0}}\right\} .
\end{array}
\right.  \label{ineq}
\end{equation}
In order to check this, we take into account that $\gamma _{0}>\frac{p}{p-1}(N-1)$ implies 
$\gamma _{0}>N$, and observe that 
\[
\beta _{0}+\xi =\frac{\gamma _{0}-p}{p(N-1)-(p-1)\gamma _{0}}\xi +\frac{p\alpha
_{0}+pN-\left( (p-1)\gamma _{0}+p\right) \beta _{0}}{p(N-1)-(p-1)\gamma _{0}}%
\Longleftrightarrow \xi =\frac{\alpha _{0}+\left( 1-\beta _{0}\right) N}{%
N-\gamma _{0}}. 
\]
Accordingly, we distinguish three subcases: \medskip 

\noindent (I) $\frac{\alpha _{0}+\left( 1-\beta _{0}\right) N}{N-\gamma _{0}}%
\geq 1-\beta _{0}$, i.e., $\alpha _{0}\leq -\gamma _{0}\left( 1-\beta
_{0}\right) =\alpha _{1}$;$\medskip $

\noindent (II) $\frac{\alpha _{0}+\left( 1-\beta _{0}\right) N}{N-\gamma _{0}%
}\leq \max \left\{ 0,\frac{1-p\beta _{0}}{p}\right\} $, i.e., 
\[
\alpha _{0}+\left( 1-\beta _{0}\right) N\geq \left( N-\gamma _{0}\right)
\max \left\{ 0,\frac{1-p\beta _{0}}{p}\right\} =\min \left\{ 0,\left(
N-\gamma _{0}\right) \frac{1-p\beta _{0}}{p}\right\} , 
\]
i.e., 
\[
\alpha _{0}\geq \min \left\{ 0,\left( N-\gamma _{0}\right) \frac{1-p\beta
_{0}}{p}\right\} -\left( 1-\beta _{0}\right) N=\min \left\{ \alpha
_{2},\alpha _{3}\right\} ; 
\]

\noindent (III) $\max \left\{ 0,\frac{1-p\beta _{0}}{p}\right\} <\frac{%
\alpha _{0}+\left( 1-\beta _{0}\right) N}{N-\gamma _{0}}<1-\beta _{0}$,
i.e., $\alpha _{1}<\alpha _{0}<\min \left\{ \alpha _{2},\alpha _{3}\right\}
.\medskip $

\noindent \emph{Subcase (I).\smallskip }

\noindent Since $\xi \leq 1-\beta _{0}$ implies 
\begin{eqnarray*}
&& 
\max \left\{ \beta _{0}+\xi ,\frac{\gamma _{0}-p}{p(N-1)-(p-1)\gamma _{0}}\xi 
+ \frac{p\alpha _{0}+pN-\left( (p-1)\gamma _{0}+p\right) \beta _{0}}{p(N-1)-(p-1)\gamma_{0}}\right\} \\
& = &
\frac{\gamma _{0}-p}{p(N-1)-(p-1)\gamma _{0}}\xi +\frac{p\alpha_{0}
+ pN-\left( (p-1) \gamma _{0}+p\right) \beta _{0}}{p(N-1)-(p-1)\gamma _{0}},
\end{eqnarray*}
%$
%\max \left\{ \beta _{0}+\xi ,\frac{\gamma _{0}-p}{p(N-1)-(p-1)\gamma _{0}}\xi +%
%\frac{p\alpha _{0}+pN-\left( (p-1)\gamma _{0}+p\right) \beta _{0}}{p(N-1)-(p-1)\gamma
%_{0}}\right\}=\frac{\gamma _{0}-p}{p(N-1)-(p-1)\gamma _{0}}\xi +\frac{p\alpha
%_{0}+pN-\left( (p-1) \gamma _{0}+p\right) \beta _{0}}{p(N-1)-(p-1)\gamma _{0}}, 
%$
the inequalities (\ref{ineq}) become 
\[
\left\{ 
\begin{array}{l}
\max \left\{ 0,\frac{1-p\beta _{0}}{p}\right\} \leq \xi \leq 1-\beta
_{0}\medskip \\ 
\frac{q_{1}}{p}>\frac{\gamma _{0}-p}{p(N-1)-(p-1)\gamma _{0}}\xi +\frac{p\alpha
_{0}+pN-\left( (p-1)\gamma _{0}+p\right) \beta _{0}}{p(N-1)-(p-1)\gamma _{0}},
\end{array}
\right. 
\]
i.e., 
\[
\left\{ 
\begin{array}{l}
\max \left\{ 0,\frac{1-p\beta _{0}}{p}\right\} \leq \xi \leq 1-\beta
_{0}\medskip \\ 
q_{1}\frac{p(N-1)-(p-1)\gamma _{0}}{p\left( \gamma _{0}-p\right) }-\frac{p^2 \alpha
_{0}+p^2 N-p\left( (p-1)\gamma _{0}+p\right) \beta _{0}}{p\left( \gamma
_{0}-p\right) } <\xi ,
\end{array}
\right. 
\]
which, since $\max \left\{ 0,\frac{1-p\beta _{0}}{p}\right\} \leq 1-\beta
_{0}$ is clearly true, has a solution $\xi $ if and only if 
\[
q_{1}\frac{p(N-1)-(p-1)\gamma _{0}}{p\left( \gamma _{0}-p\right) }-\frac{p^2 \alpha
_{0}+p^2 N-p\left( (p-1)\gamma _{0}+p\right) \beta _{0}}{p\left( \gamma
_{0}-p\right) }  <1-\beta _{0}, 
\]
i.e., 
\begin{eqnarray*}
q_{1}&>&\frac{p^2\alpha _{0}+p^2N-p\left( (p-1)\gamma _{0}+p\right) \beta _{0}
+p\left(\gamma _{0}-p\right) \left( 1-\beta _{0}\right) }{p(N-1)-(p-1)\gamma _{0}}\\
&=&
p\frac{p\alpha _{0}+pN-p+\gamma _{0}\left( 1-p\beta _{0}\right) }{p(N-1)-(p-1)\gamma _{0}}. 
\end{eqnarray*}
%\medskip 

\noindent \emph{Subcase (II).\smallskip }

\noindent Since $\xi \geq \max \left\{ 0,\frac{1-p\beta _{0}}{p}\right\} $
implies $\max \left\{ \beta _{0}+\xi ,\frac{\gamma _{0}-p}{p(N-1)-(p-1)\gamma _{0}}%
\xi +\frac{p\alpha _{0}+pN-\left( (p-1)\gamma _{0}+p\right) \beta _{0}}{%
p(N-1)-(p-1)\gamma _{0}}\right\} =\beta _{0}+\xi $, the inequalities (\ref{ineq})
become 
\[
\left\{ 
\begin{array}{l}
\max \left\{ 0,\frac{1-p\beta _{0}}{p}\right\} \leq \xi \leq 1-\beta
_{0}\medskip \\ 
\xi <\frac{q_{1}}{p}-\beta _{0},
\end{array}
\right. 
\]
which has a solution $\xi $ if and only if $\max \left\{ 0,\frac{1-p\beta
_{0}}{p}\right\} \leq \frac{q_{1}}{p}-\beta _{0}$, i.e., $q_{1}>\max \left\{
1,p\beta _{0}\right\} $.$\medskip $

\noindent \emph{Subcase (III).\smallskip }

\noindent We take $\xi =\frac{\alpha _{0}+\left( 1-\beta _{0}\right) N}{%
N-\gamma _{0}}$ and thus the inequalities (\ref{ineq}) are equivalent just
to 
\[
\frac{q_{1}}{p}>\beta _{0}+\frac{\alpha _{0}+\left( 1-\beta _{0}\right) N}{%
N-\gamma _{0}}=\frac{\alpha _{0}+N-\gamma _{0}\beta _{0}}{N-\gamma _{0}}. 
\]

\end{document}